\documentclass[a4paper, 11pt]{article}
\usepackage{amsmath, amsthm, amssymb, amscd, mathrsfs, amsfonts,graphicx, tikz, tikz-cd, pgf, hyperref, mathtools, scalerel, float, multicol, xcolor, url, enumitem, xcolor}
\usepackage{subcaption}
\usepackage{lmodern}
\usepackage[margin=2.5cm]{geometry}
\usepackage[utf8]{inputenc}
\usepackage{csquotes}
\usetikzlibrary{shapes.geometric}
\usetikzlibrary{angles}
\usepackage[T1]{fontenc}
\usetikzlibrary{calc}
\usepackage{authblk}
\usepackage[font=small,labelfont=bf]{caption}
\makeatletter \renewcommand{\fnum@figure}{Fig. \thefigure} \makeatother
\newtheorem{theorem}{Theorem}[section]

\newtheorem{definition}[theorem]{Definition}

\newtheorem{lemma}[theorem]{Lemma}

\newtheorem{proposition}[theorem]{Proposition}
\newtheorem{remark}[theorem]{Remark}

\numberwithin{equation}{section}
\numberwithin{figure}{section}

\definecolor{centralfill}{RGB}{62,78,92}
\definecolor{firstfill}{RGB}{83,155,211}
\definecolor{secondfill}{RGB}{248,176,79}

\tikzset{
 cellgrid/.style={draw=white!72, line width=0.18pt},
 copyborder/.style={draw=black, line width=0.85pt,
                   line cap=round, line join=round},
  copylabel/.style={font=\scriptsize\bfseries, text=black,
                   fill=white, fill opacity=.72, text opacity=1,
                   inner sep=1pt, rounded corners=.5pt}
}

\definecolor{originalblue}{RGB}{0,114,178}
\definecolor{dualvermillion}{RGB}{213,94,0}

\begin{document}
	\title{Unboundedness of the Heesch Number for Hyperbolic Convex Monotiles}
	\author{%
Arun Maiti 
}
\providecommand{\keywords}[1]{\textit{Keywords:} #1}
	\maketitle
	\begin{abstract}
A homogeneous (also known as semi-regular) tiling is an edge-to-edge tiling by regular polygons in which every vertex has the same cyclic type. We resolve the Heesch problem in the hyperbolic plane both for such tilings and for convex monotiles, in the latter case without any edge-to-edge restriction. We also construct an infinite family of weakly aperiodic convex monotiles whose interior angles are rational multiples of $\pi$.
\end{abstract}

	
	\section{Introduction}\label{intro}
	The domino problem, a central decision problem in tiling theory, asks whether the integer lattice $\mathbb{Z}^2$ can be tiled with unit-square tiles subject to local color-matching rules. 
	Introduced by Wang in 1961 as a geometric reformulation of satisfiability \cite{HW61}, the problem has since inspired numerous variants and generalizations across diverse settings \cite{J10, JK07}. 
	\par
	In the hyperbolic plane ($\mathbb{H}^2$), an analogue of the problem can be formulated as follows: given a finite collection $\mathcal{F}$ of tiles (polygons with geodesic sides) in $\mathbb{H}^2$, called a \textit{protoset}, together with a tiling rule, is it possible to tile $\mathbb{H}^2$ by isometric copies of the tiles in $\mathcal{F}$? The problem for arbitrary finite protosets was proved to be undecidable independently around the same time by Kari \cite{JK07} and Margenstern \cite{MM08}. The Euclidean analogue was established much earlier by Berger \cite{RB66}. The decidability of the domino problem in any of these settings is closely linked to the existence of aperiodic protosets and the Heesch problem.
	\par
	\subsection{Homogeneous Tilings}
	Here, we consider a constrained variant of the domino problem where protosets consist of an arbitrary finite set of regular polygons together with the tiling rule that the cyclic sequence of polygons around every vertex is the same and that polygons meet edge-to-edge. By duality, this problem is closely related to the domino problem for protosets consisting of a single convex tile (monotile).
\par
	The \textit{type} of a vertex in a tiling is defined to be the cyclic sequence of the \textit{sizes} (number of sides) of the polygons incident to a vertex. A vertex type and its mirror image are considered to be the same. A tiling of a surface is called \textit{homogeneous} (also known as Archimedean and semi-regular) if all vertices have the same type. 
	The \textit{type} of a homogeneous tiling is defined to be the type of its vertices. The \enquote{angle-sum} of a cyclic tuple
	 $\mathfrak{k}=[k_1, k_2, \cdots, k_d]$ with $3 \leq k_i < \infty$ for all $i$, is defined by 
	\[ \vartheta(\mathfrak{k})= \sum_{i=1}^d \frac{k_i-2}{k_i} .\]

For a homogeneous tiling of $\mathbb{H}^2$ of type $\mathfrak{k}$ by regular polygons, the side lengths of the polygons are uniquely (up to isometry) determined by the type $\mathfrak{k}$ whenever $k_i < \infty$ and $\vartheta(\mathfrak{k}) > 2$ \cite{DG18}. This allows us to view homogeneous tilings of $\mathbb{H}^2$ as topological tilings of the plane. Conversely, a homogeneous tiling of the plane of type $\mathfrak{k}$ satisfying $\vartheta(\mathfrak{k}) > 2$ (or $\vartheta(\mathfrak{k}) = 2$) can be realized as a geometric tiling of $\mathbb{H}^2$ (respectively of $\mathbb{E}^2$); see \cite{DG18}.

In light of this, we define the combinatorial construction of such a tiling inductively. The \emph{$0$-layer patch} $X_0$ of type $\mathfrak{k}$ is the configuration obtained by attaching $d$ faces around a distinguished vertex $v_0$ with sizes $k_1, \dots, k_d$ in cyclic order. For $n \ge 1$, the \emph{$n$-layer patch} $X_n$ of type $\mathfrak{k} $ is obtained from $X_{n-1}$ by attaching faces along the boundary $\partial X_{n-1}$ such that the vertex-type at each boundary vertex matches $\mathfrak{k}$.

A cyclic tuple $\mathfrak{k}$ \emph{admits a tiling} if this construction can be extended for all $k \in \mathbb{N}$; the resulting full tiling is given by the union $\mathcal{T} = \bigcup_{k=0}^{\infty} X_k$.

We define the \emph{$n$-th corona} $\mathcal{C}_n$ as the set of faces added at the $n$-th stage
\[
\mathcal C_n:=\{f\in X_n\mid f\notin X_{n-1}\},
\]
and denote its outer boundary by
\(\partial\mathcal C_n=\partial X_n\).

	\subsection{The Heesch Problem} The study of local-to-global obstructions in a tiling is encapsulated by the Heesch number, which quantifies the extent to which a local configuration can be extended before a combinatorial or geometric conflict arises.
\begin{definition}[Heesch Number of a Cyclic Tuple]
For a cyclic tuple $\mathfrak{k}$ with $\vartheta(\mathfrak{k}) > 2$, the \textit{Heesch number} $\mathcal{H}(\mathfrak{k})$ is the maximal non-negative integer $r$ such that $\mathfrak{k}$ admits an $r$-layer patch of type $\mathfrak{k}$. If $\mathfrak{k}$ admits a full tiling $\mathcal{T}$ of the plane, we define $\mathcal{H}(\mathfrak{k}) = \infty$.
\end{definition}
While the cyclic tuple describes the vertex-type constraints, we may also define this property for a single geometric shape (a tile). In this case, the layers are formed by congruent copies of the tile itself.
\begin{definition}[Heesch Number of a Prototile]
The \textit{Heesch number} $\mathcal{H}(T)$ of a prototile $T$ is the maximal non-negative integer $r$ such that there exists a patch $X_r$ consisting of $T$ (the $0$-th corona) completely surrounded by $r$ successive coronas $\mathcal{C}_1, \dots, \mathcal{C}_r$ of tiles congruent to $T$.
\end{definition}
 The \textit{Heesch problem} asks which integers can occur as Heesch numbers for a given class of tiles. The study of Heesch numbers measures how closely finite patches can approximate a tiling of the hyperbolic plane, links finite local patterns to infinite global structures, and deepens our understanding of undecidability in tiling theory. In \cite{AM23}, the author showed that Heesch numbers are bounded for finite sets of regular polygons that do not admit a tiling of $\mathbb{H}^2$. 
\par
This domino problem for homogeneous tilings has been explored by numerous authors \cite{DG18, GS, AM20, DR08}, particularly for cyclic tuples of low degree, though the general problem remains open. While an explicit description of cyclic tuples admitting homogeneous tilings up to degree $6$ was provided in \cite{SS22}, the construction of such tilings is typically restricted to layer-by-layer growth or standard operations on known vertex-transitive tilings. To the author's knowledge, every cyclic tuple currently known to not admit a tiling also fails to admit a patch beyond two layers--implying $\mathcal{H}(\mathfrak{k}) \le 2$. Our investigation reveals that the local extension behavior of homogeneous tilings is substantially richer than previously observed.

\begin{theorem} \label{heeschn} For every positive integer $n \in \mathbb{N}$, there exists a cyclic tuple $\mathfrak{k}_n$ of length $9n^2+25n$ such that $\mathcal{H}(\mathfrak{k}_n) = n$. 
	\end{theorem}
	In the proof (presented in \S \ref{Heesch}) the cyclic tuple $\mathfrak{k}_n$ is built inductively from $\mathfrak{k}_{n-1}$ by adjoining a complementary block $\tilde{\mathfrak{k}}_{n}$. The tuple $\mathfrak{k}_{n-1}$ admits an $n-1$-layer patch and an almost full $n$-layer patch with the sole obstruction around a single odd-sided face. 
	
 The role of this added block $\bar{\mathfrak{k}}_{n}$ is twofold: it resolves the unique local obstruction that appears near an odd-sided face in corona $\mathcal{C}_{n-1}$, while at the same time creating a new forced obstruction around an adjacent odd-sided face one layer farther out. Iterating this local obstruction-propagation mechanism produces cyclic tuples with arbitrarily large finite Heesch number.
	\\
	In~\cite{AT10}, A.~S.~Tarasov constructed hyperbolic monotiles with every prescribed positive Heesch number by adding suitable notches and ridges to regular polygons; the resulting tiles are non-convex. Our passage to convex monotiles uses canonical duality: joining the incentres of consecutive faces  transforms every homogeneous $n$-layer patch into an edge-to-edge $n$-layer patch of tiling by single convex tile. The converse, however, need not hold, particularly when non-edge-to-edge arrangements are permitted. To control such additional configurations, in Theorem~\ref{convexthm}, we replace $\mathfrak{k}_n$ by a larger all-prime tuple $\mathfrak{p}_n$, obtained from the same obstruction-propagation scheme and still satisfying $\mathcal H(\mathfrak{p}_n)=n$. An elementary reciprocal-prime argument forces every completed vertex in a patch by its canonical dual $Q_n$ to consist of corners of a single type and excludes completed $T$-junctions. Consequently, $Q_n$ can admit at most one corona beyond those arising by duality from homogeneous patches. In fact, we prove $\mathcal H(Q_n)=n+1$, where the last corona is non-edge-to-edge. These tiles yields the convex monotiles with any prescribed Heesch number.

\par

\subsection{Aperiodicity in the Hyperbolic Plane}

Periodicity in the hyperbolic plane has two distinct forms. A tiling is \emph{strongly periodic} if the quotient by its symmetry group is compact, and \emph{weakly periodic} if its symmetry group contains an infinite cyclic subgroup~\cite{GS05}. Accordingly, a protoset is called \emph{weakly aperiodic} if it admits no strongly periodic tiling, and \emph{strongly aperiodic} if it admits no weakly periodic tiling. These notions coincide in the Euclidean plane. Aperiodic tile sets play an important role in the domino and periodic domino problems~\cite{B64,MM09,J10}; in particular, Margenstern proved the undecidability of the strongly periodic domino problem for general hyperbolic protosets~\cite{MM09}.

Most classical tilings by regular polygons are vertex-transitive and hence homogeneous~\cite{GS79}. Examples admitting only tilings with several vertex orbits appeared more recently in~\cite{AM20} and in informal notes of Marek \v{C}trn\'{a}ct~\cite{marek}. In~\cite{AM23}, the present author used a double-counting argument involving incidences between triangles and pentagons to obstruct strongly periodic tilings of types $[3,5,k_3,k_4]$ for $10\leq k_3\leq k_4$, with $k_3,k_4\neq11$.

Adapting that argument to homogeneous tilings, we show that the cyclic type
\[
[3,5,k,5,l,5,m,5,l,5,k,5,l,5]
\]
admits a homogeneous tiling but no strongly periodic one whenever $k,l,$ and $m$ are distinct integers at least $6$. This gives an infinite family of weakly aperiodic homogeneous vertex types. When $3,5,k,l,$ and $m$ are distinct primes, the canonical dual is moreover a weakly aperiodic convex monotile, even without an edge-to-edge assumption. Its interior angles, and hence its area, are rational multiples of $\pi$. This contrasts with the convex aperiodic tiles with strictly irrational angles constructed by Margulis and Mozes~\cite{MM98}; in the Euclidean plane, Rao proved that every convex tile admitting a tiling also admits a periodic one~\cite{MR17}. Together with the unboundedness of Heesch numbers established above, these examples suggest that the general domino problem for homogeneous hyperbolic tilings may be undecidable.

\section{Preliminaries} \label{localext}

\begin{definition}[Fan around a Vertex] \label{fan}
Let $P$ be a face in a patch or tiling of type $\mathfrak{k}=[k_1, k_2, \dots, k_d]$ where $P$ corresponds to the size $k_1$, and let $v$ be a vertex of $P$. A \emph{fan} around the vertex $v$ anchored at $P$ is a configuration of $d$ faces, denoted $F_v = (P, f_1, f_2, \dots, f_{d-1})$, which are incident to $v$ and meet edge-to-edge in the clockwise order dictated by the cyclic tuple. 

A \emph{partial fan} around the vertex $v$ anchored at $P$ is a contiguous subsequence of $F_v$ beginning with $P$ (e.g., $(P, f_1, \dots, f_m)$ for $m < d-1$), that can be extended to a full fan around $v$.
\end{definition}

\begin{definition}[Neighborhood of a Face] Let $\mathfrak{k}$ be a cyclic tuple. In a patch or tiling of type $\mathfrak{k}$, the \emph{neighborhood} of a face $P$ is the cyclic configuration of faces incident to the vertices of $P$ such that every vertex in the configuration has type $\mathfrak{k}$.

For homogeneous tiling, the neighborhood of a face is largely determined by the faces sharing edges with the central face. Therefore, for brevity, we will represent the neighborhood of a face solely by the cyclic tuple of its edge-adjacent faces, enclosed in square brackets (e.g., $[a, b, c, \dots]$).
 \end{definition}
 
 \begin{definition}[Partial Neighborhood around a Face]
A \emph{partial neighborhood} of length $m$ around a face $P$ is a configuration obtained by forming fans around $m$ consecutive vertices of $P$. As with neighborhoods, we represent a partial neighborhood solely by the ordered tuple of faces sharing edges with $P$, enclosed in parentheses (e.g., $(a,b,c,\dots)$).

\end{definition}

 \begin{definition}[Local Patch around a Face]
Let $P$ be a face in a patch or tiling. The \emph{$0$-layer patch} around $P$, denoted $X_0(P)$, is defined as the face $P$ itself. Inductively, for $m \ge 1$, the \emph{$m$-layer patch} $X_m(P)$ is the configuration obtained by attaching full neighborhoods to all faces in the boundary of $X_{m-1}(P)$. 
\end{definition}

\subsection{Local Extension Mechanisms} 
Let $\mathfrak{k}=[k_1, k_2, \dots, k_d]$, and $S = (k_1, k_2, \dots, k_d)$ be an ordered sequence representing a fan anchored at a face $P$. The \emph{reflection} of $S$, denoted $S^*$, is the sequence traversed in the reverse orientation: $S^* = (k_1, k_d, k_{d-1}, \dots, k_2)$. 

Let the vertices of $P$ be $v_1, v_2, \dots, v_{k_1}$ in the clockwise direction. One can always construct a valid partial neighborhood of length $2$ around $P$ by placing the fan $S$ at $v_1$ and its reflection $S^*$ at $v_2$. By continuing this alternating assignment--placing $(S^*)^* = S$ at $v_3$, $S^*$ at $v_4$, and so on--one can systematically build a larger partial neighborhood.

If $k_1$ is even, this alternating process closes consistently along the final shared edge $v_{k_1}v_1$, producing a full neighborhood around $P$. However, if $k_1$ is odd, this specific process fails to close consistently. The alternating pattern dictates that the fan at $v_{k_1}$ is $S$, which creates an incidence obstruction because the fan $S$ at $v_{k_1}$ cannot validly share the final edge $v_{k_1}v_1$ with the identical fan $S$ at $v_1$.

The following lemma shows that even-sized faces not only admit full neighborhoods, but their presence also facilitates the formation of layers during the layer-by-layer construction of a homogeneous tiling.

\begin{lemma}\label{even} 
Let $P$ be a face of even size $k$ in the $n$-th corona $\mathcal{C}_n$ of an $n$-layer patch of type $\mathfrak{k}$.
\begin{enumerate}
 \item[(a)] If $k \geq 8$, then $P$ always admits a full neighborhood, regardless of the neighborhoods chosen for the other faces in $\mathcal{C}_n$.
 
 \item[(b)] If $k \geq 12$, $n \geq 1$, and no odd-even-odd subtuple occurs in $\mathfrak{k}$, then this full neighborhood can be chosen such that it extends to a complete $2$-layer patch around $P$.
\end{enumerate}
\end{lemma}

\begin{proof}[Proof of (a)]
	Let $P$ be a face of even size $k \ge 8$ in $\mathcal{C}_n$. $P$ shares either a single vertex or a single edge with the boundary $\partial \mathcal{C}_{n-1}$. This imposes a partial neighborhood for $P$ of length at most $2$; see Fig. \ref{fig:nbdflip8}. Furthermore, for arbitrary choices of neighborhoods for all other faces in $\mathcal{C}_n$, the two adjacent faces of $P$ (within the same corona) can constrain at most one additional fan each at the neighboring vertices of $P$. Consequently, we obtain a partial neighborhood of $P$ of length at most $4$.
	
	We only need to consider the case of length $4$, as a partial neighborhood of length $3$ can always be trivially extended to length $4$. If the size of $P$ is exactly $8$, one can simply reverse the sequence of $4$ fans across the remaining $4$ vertices to construct a full neighborhood around the $8$-gon. For example, a partial neighborhood $(a, b, c, d, e)$ extends to a full cyclic neighborhood $[a, b, c, d, e, d, c, b]$, where the terminal faces $a$ and $e$ act as the axes of reflection. This process is illustrated in Fig. \ref{fig:nbdflip8}, assuming a partial neighborhood consisting of fans around four consecutive vertices $v_1, v_2, v_3,$ and $v_4$.

For faces of larger even sizes, in addition to reversing the sequence of $4$ consecutive fans, one can systematically reflect a single boundary fan (about the boundary face) an even number of times to fill the remaining vertices and close the neighborhood. This is illustrated in Fig. \ref{fig:nbdflip10} for a 10-gon with initial fans at vertices $v_1$, $v_2$, $v_3$, and $v_4$. The partial neighborhood is first reversed about the face $e$, and additionally, the fan at vertex $v_1$ is reflected twice more: first about the face $a$, and then about the face $b$, thereby completing the full neighborhood.

	\begin{figure}[H] 
	\centering
	\begin{minipage}{0.57\textwidth}
 \tikzstyle{ver}=[]
 \tikzstyle{edge} = [draw,thick,-]
 \centering
 	
	\begin{tikzpicture}[
    scale=.4,
    transform shape,
    every node/.style={scale=1.5}
]
		% Define styles
		\tikzstyle{edge}=[thick]
		\tikzstyle{ver}=[font=\small]
		
		% 1. VERTEX COORDINATES
		% Top row (Left to right)
		\coordinate (v1) at (-5, 0);
		\coordinate (v8) at (-3, 0);
		\coordinate (v7) at (-1, 0);
		\coordinate (v6) at (1, 0);
		\coordinate (v5) at (3, 0);
		\coordinate (v4) at (5, 0);
		
		% Middle row
		\coordinate (v2) at (-2.5, -2.5);
		\coordinate (v3) at (2.5, -2.5);
		
		% Bottom row (New vertices on \partial X_{n-1})
		\coordinate (u2) at (-3.5, -5);
		\coordinate (u3) at (3.5, -5);
		
		% 2. HORIZONTAL BOUNDARY LINES & LABELS
		% Topmost boundary line (\partial X_{n+2})
		\draw (-8.5, 2.5) -- (8.5, 2.5) node[right] {$\partial X_{n+1}$};
		
		% The line containing v1 to v4 (\partial X_{n+1})
		\draw (-8.5, 0) -- (8.5, 0) node[right] {$\partial X_{n}$};
		
		% The line containing v2 to v3 (\partial X_n)
		\draw (-8.5, -2.5) -- (8.5, -2.5) node[right] {$\partial X_{n-1}$};
		
		% The bottommost line (\partial X_{n-1})
		\draw (-8.5, -5) -- (8.5, -5) node[right] {$\partial X_{n-2}$};
		
		% 3. POLYGON EDGES (The main body)
		% Top horizontal edges (labels placed below to be inside the octagon)
		\draw[edge] (v1) -- (v8) node[midway, above, yshift=0.5cm] {$a$};
		\draw[edge] (v8) -- (v7) node[midway, above, yshift=0.5cm] {$b$};
		\draw[edge] (v7) -- (v6) node[midway, above, yshift=0.5cm] {$c$};
		\draw[edge] (v6) -- (v5) node[midway, above, yshift=0.5cm] {$d$};
		\draw[edge] (v5) -- (v4) node[midway, above, yshift=0.5cm] {$e$};
		
		% Connecting diagonal sides
		\draw[edge] (v1) -- (v2) node[midway, above left, xshift=-0.9cm] {$b$};
		\draw[edge] (v4) -- (v3) node[midway, above left, xshift=1.5cm] {$d$};
		
		% Bottom horizontal edge
		\draw[edge] (v2) -- (v3) node[midway, below, yshift=-1cm] {$c$};
	
		% 4. SPIKES AND DOTTED CURVES
		% Top 6 spikes (Centers at v1, v8, v7, v6, v5, v4 pointing UPWARD to \partial X_{n+2})
		\foreach \v in {v1, v8, v7, v6, v5, v4} {
			\draw (\v) -- ++(-0.8, 2.5);
			\draw (\v) -- ++(0.8, 2.5);
			\draw[thick, dotted] (\v) ++(-0.4, 1.25) to[bend left=30] ++(0.8, 0);
		}
		
		% Middle 2 spikes (Centers at v2, v3, pointing UPWARD to \partial X_{n+1})
		% For v2
		\draw (v2) -- (-7.5, 0);
		\draw[thick, dotted] (-5, -1.25) to[bend left=30] (-3.75, -1.25);
		
		% For v3
		\draw (v3) -- (7.5, 0);
		\draw[thick, dotted] (3.75, -1.25) to[bend left=30] (5, -1.25);
		
		% Bottom 2 spikes (Centers at u2, u3 on \partial X_{n-1}, pointing UPWARD to \partial X_n)
		% For u2 (connected to v2)
		\draw (u2) -- (v2);
		\draw (u2) -- (-4.5, -2.5);
		\draw[thick, dotted] (-4, -3.75) to[bend left=30] (-3.0, -3.75);
		
		% For u3 (connected to v3)
		\draw (u3) -- (5, -2.5);
		\draw (u3) -- (6.5, -2.5);
		\draw[thick, dotted] (4.3, -3.7) to[bend left=30] (5, -3.75);
		
		% 5. CENTRAL NUMBER
		\node at (0, -1.25) {\Huge \textbf{8}};
		
		% 6. VERTEX LABELS
		\node[below left] at (v1) {$v_1$};
		\node[below] at (v8) {$v_8$};
		\node[below] at (v7) {$v_7$};
		\node[below] at (v6) {$v_6$};
		\node[below] at (v5) {$v_5$};
		\node[below right] at (v4) {$v_4$};
		
		% Shifted v2 and v3 labels slightly so they don't overlap the new lower lines
		\node[below right] at (v2) {$v_2$};
		\node[below left] at (v3) {$v_3$};
		
	\end{tikzpicture}
 	\captionof{figure}{Completing a partial neighborhood of 8-gon}
			\label{fig:nbdflip8}	
\end{minipage}%
\begin{minipage}{0.38\textwidth}
 \tikzstyle{ver}=[]
 \tikzstyle{edge} = [draw,thick,-]
 \centering
 \begin{tikzpicture}[scale=0.6, transform shape] 
 
 % 1. EXACT VERTEX COORDINATES FOR 10-GON
 \coordinate (v9) at (-0.57, 1.76);
 \coordinate (v8) at (0.57, 1.76);
 \coordinate (v7) at (1.50, 1.09);
 \coordinate (v6) at (1.85, 0.00);
 \coordinate (v5) at (1.50, -1.09);
 \coordinate (v4) at (0.57, -1.76);
 \coordinate (v3) at (-0.57, -1.76);
 \coordinate (v2) at (-1.50, -1.09);
 \coordinate (v1) at (-1.85, 0.00);
 \coordinate (v10) at (-1.50, 1.09);

 % 2. THE 10-GON EDGES
 \draw[edge] (v1) -- (v2) -- (v3) -- (v4) -- (v5) -- (v6) -- (v7) -- (v8) -- (v9) -- (v10) -- cycle;

 % 3. SYMMETRICAL WHISKERS (Length = 0.65, Spread = ±28 degrees from radial center)
 \draw[edge] (v9) -- ++(80:0.65); \draw[edge] (v9) -- ++(136:0.65);
 \draw[edge] (v8) -- ++(44:0.65); \draw[edge] (v8) -- ++(100:0.65);
 \draw[edge] (v7) -- ++(8:0.65); \draw[edge] (v7) -- ++(64:0.65);
 \draw[edge] (v6) -- ++(-28:0.65); \draw[edge] (v6) -- ++(28:0.65);
 \draw[edge] (v5) -- ++(-64:0.65); \draw[edge] (v5) -- ++(-8:0.65);
 \draw[edge] (v4) -- ++(-100:0.65); \draw[edge] (v4) -- ++(-44:0.65);
 \draw[edge] (v3) -- ++(-136:0.65); \draw[edge] (v3) -- ++(-80:0.65);
 \draw[edge] (v2) -- ++(188:0.65); \draw[edge] (v2) -- ++(244:0.65); 
 \draw[edge] (v1) -- ++(152:0.65); \draw[edge] (v1) -- ++(208:0.65);
 \draw[edge] (v10) -- ++(116:0.65); \draw[edge] (v10) -- ++(172:0.65);

 % 4. DOTTED CURVES INSIDE THE ANGLES (Radius = 0.5)
 \draw[thick, dotted] (v9) ++(80:0.5) arc (80:136:0.5);
 \draw[thick, dotted] (v8) ++(44:0.5) arc (44:100:0.5);
 \draw[thick, dotted] (v7) ++(8:0.5) arc (8:64:0.5);
 \draw[thick, dotted] (v6) ++(-28:0.5) arc (-28:28:0.5);
 \draw[thick, dotted] (v5) ++(-64:0.5) arc (-64:-8:0.5);
 \draw[thick, dotted] (v4) ++(-100:0.5) arc (-100:-44:0.5);
 \draw[thick, dotted] (v3) ++(-136:0.5) arc (-136:-80:0.5);
 \draw[thick, dotted] (v2) ++(188:0.5) arc (188:244:0.5); 
 \draw[thick, dotted] (v1) ++(152:0.5) arc (152:208:0.5);
 \draw[thick, dotted] (v10) ++(116:0.5) arc (116:172:0.5);

 % 5. CENTRAL NUMBER
 \node[ver] at (0,0){$10$};

 % 6. VERTEX LABELS (Pulled slightly inward to tuck into corners)
 \node[ver] at (-0.46, 1.43){$v_9$};
 \node[ver] at (0.46, 1.43){$v_8$};
 \node[ver] at (1.21, 0.88){$v_7$};
 \node[ver] at (1.50, 0){$v_6$};
 \node[ver] at (1.21, -0.88){$v_5$};
 \node[ver] at (0.46, -1.43){$v_4$};
 \node[ver] at (-0.46, -1.43){$v_3$};
 \node[ver] at (-1.21, -0.88){$v_2$};
 \node[ver] at (-1.50, 0){$v_1$};
 \node[ver] at (-1.21, 0.88){$v_{10}$};

 % 7. EDGE LABELS (Outside the edges, matched to your original pattern)
 \node[ver] at (0, 2.05){$a$}; % Top (v9 to v8)
 \node[ver] at (1.20, 1.66){$b$}; % Top-right (v8 to v7)
 \node[ver] at (1.95, 0.63){$c$}; % Right-top (v7 to v6)
 \node[ver] at (1.95, -0.63){$d$}; % Right-bottom (v6 to v5)
 \node[ver] at (1.20, -1.66){$e$}; % Bottom-right (v5 to v4)
 \node[ver] at (0, -2.05){$d$}; % Bottom (v4 to v3)
 \node[ver] at (-1.20, -1.66){$c$}; % Bottom-left (v3 to v2)
 \node[ver] at (-1.95, -0.63){$b$}; % Left-bottom (v2 to v1)
 \node[ver] at (-1.95, 0.63){$a$}; % Left-top (v1 to v10)
 \node[ver] at (-1.20, 1.66){$b$}; % Top-left (v10 to v9)

 \end{tikzpicture}
 	\captionof{figure}{Completing a partial neighborhood of an even-gon}
			\label{fig:nbdflip10}
\end{minipage}	
\end{figure}

\end{proof}

\begin{proof}[Proof of (b)] To prove the stronger statement for $k \ge 12$, $n \geq 1$, observe that simply reversing the existing partial neighborhood $(a, b, c, d, e)$ of $P$ together with the neighborhoods of faces $ b, c$, and $d$ may induce a partial neighborhood on the boundary faces $a$ and $e$ (which act as the terminal faces for the reversal process) that cannot be extended to a full neighborhood if one or both of them are odd-sided. 

Suppose both $a$ and $e$ are odd-sided faces. We can extend the partial neighborhood $(a, b, c, d, e)$ of $P$ at both ends by one additional fan by forming full neighborhoods around the faces $a$ and $e$. This extends the partial neighborhood of $P$ to the form $(a', a, b, c, d, e, e')$; see Fig. \ref{fig:nbdflip12} for an illustration. Since no odd-even-odd subtuple occurs in $\mathfrak{k}$--meaning $\mathfrak{k}$ does not contain a subtuple of the form $(k_i, k_j, k_l)$ where both $k_i$ and $k_l$ are odd--two odd-sided faces can never appear consecutively in the neighborhood of any even-sided face. Consequently, the new boundary faces $a'$ and $e'$ cannot be odd-sided.

We can then safely use $a'$ and $e'$ as the new terminal faces to initiate the reversal process established above. By reversing this sequence along with the neighbors of $a, b, c, d$ and $e$, we form a full neighborhood of $P$ such that all faces except $a'$ and $e'$ possess full neighborhoods. Lastly, we form full neighborhoods around the even-sided faces $a'$ and $e'$ using the same reversal method. Thus, we successfully construct a complete $2$-layer patch around $P$. Note that this extended partial neighborhood involves $7$ faces spanning $6$ consecutive vertices. Reversing this sequence (and applying further reflections if $k > 14$) requires $P$ to have at least $12$ vertices (i.e., $k \geq 12$).

A similar extension argument applies if only one of $a$ or $e$ is odd-sided. This establishes the required extension to a complete $2$-layer patch around $P$.

\begin{figure}[H] 
 \tikzstyle{ver}=[]
 \tikzstyle{edge} = [draw,thick,-]
 \centering
 \begin{tikzpicture}[scale=0.6, transform shape] 
 
 % 1. EXACT VERTEX COORDINATES FOR 12-GON (Radius = 2.0)
 \coordinate (v10) at (0.00, 2.00);
 \coordinate (v9) at (1.00, 1.73);
 \coordinate (v8) at (1.73, 1.00);
 \coordinate (v7) at (2.00, 0.00);
 \coordinate (v6) at (1.73, -1.00);
 \coordinate (v5) at (1.00, -1.73);
 \coordinate (v4) at (0.00, -2.00);
 \coordinate (v3) at (-1.00, -1.73);
 \coordinate (v2) at (-1.73, -1.00);
 \coordinate (v1) at (-2.00, 0.00);
 \coordinate (v12) at (-1.73, 1.00);
 \coordinate (v11) at (-1.00, 1.73);

 % 2. THE 12-GON EDGES
 \draw[edge] (v1) -- (v2) -- (v3) -- (v4) -- (v5) -- (v6) -- (v7) -- (v8) -- (v9) -- (v10) -- (v11) -- (v12) -- cycle;

 % 3. SYMMETRICAL WHISKERS (Length = 0.65, Spread = \pm 28 degrees from radial center)
 \draw[edge] (v10) -- ++(62:0.65); \draw[edge] (v10) -- ++(118:0.65);
 \draw[edge] (v9) -- ++(32:0.65); \draw[edge] (v9) -- ++(88:0.65);
 \draw[edge] (v8) -- ++(2:0.65); \draw[edge] (v8) -- ++(58:0.65);
 \draw[edge] (v7) -- ++(-28:0.65); \draw[edge] (v7) -- ++(28:0.65);
 \draw[edge] (v6) -- ++(-58:0.65); \draw[edge] (v6) -- ++(-2:0.65);
 \draw[edge] (v5) -- ++(-88:0.65); \draw[edge] (v5) -- ++(-32:0.65);
 \draw[edge] (v4) -- ++(-118:0.65); \draw[edge] (v4) -- ++(-62:0.65);
 \draw[edge] (v3) -- ++(-148:0.65); \draw[edge] (v3) -- ++(-92:0.65);
 \draw[edge] (v2) -- ++(-178:0.65); \draw[edge] (v2) -- ++(-122:0.65);
 \draw[edge] (v1) -- ++(152:0.65); \draw[edge] (v1) -- ++(208:0.65);
 \draw[edge] (v12) -- ++(122:0.65); \draw[edge] (v12) -- ++(178:0.65);
 \draw[edge] (v11) -- ++(92:0.65); \draw[edge] (v11) -- ++(148:0.65);

 % 4. DOTTED CURVES INSIDE THE ANGLES (Radius = 0.5)
 \draw[thick, dotted] (v10) ++(62:0.5) arc (62:118:0.5);
 \draw[thick, dotted] (v9) ++(32:0.5) arc (32:88:0.5);
 \draw[thick, dotted] (v8) ++(2:0.5) arc (2:58:0.5);
 \draw[thick, dotted] (v7) ++(-28:0.5) arc (-28:28:0.5);
 \draw[thick, dotted] (v6) ++(-58:0.5) arc (-58:-2:0.5);
 \draw[thick, dotted] (v5) ++(-88:0.5) arc (-88:-32:0.5);
 \draw[thick, dotted] (v4) ++(-118:0.5) arc (-118:-62:0.5);
 \draw[thick, dotted] (v3) ++(-148:0.5) arc (-148:-92:0.5);
 \draw[thick, dotted] (v2) ++(-178:0.5) arc (-178:-122:0.5);
 \draw[thick, dotted] (v1) ++(152:0.5) arc (152:208:0.5);
 \draw[thick, dotted] (v12) ++(122:0.5) arc (122:178:0.5);
 \draw[thick, dotted] (v11) ++(92:0.5) arc (92:148:0.5);

 % 5. CENTRAL NUMBER
 \node[ver] at (0,0){$12$};

 % 6. VERTEX LABELS (Pulled slightly inward to tuck into corners)
 \node[ver] at (0.00, 1.60){$v_{10}$};
 \node[ver] at (0.80, 1.39){$v_9$};
 \node[ver] at (1.39, 0.80){$v_8$};
 \node[ver] at (1.60, 0.00){$v_7$};
 \node[ver] at (1.39, -0.80){$v_6$};
 \node[ver] at (0.80, -1.39){$v_5$};
 \node[ver] at (0.00, -1.60){$v_4$};
 \node[ver] at (-0.80, -1.39){$v_3$};
 \node[ver] at (-1.39, -0.80){$v_2$};
 \node[ver] at (-1.60, 0.00){$v_1$};
 \node[ver] at (-1.39, 0.80){$v_{12}$};
 \node[ver] at (-0.80, 1.39){$v_{11}$};

 % 7. EDGE LABELS (Outside the edges)
 \node[ver] at (0.58, 2.17){$b$}; % v11 to v10
 \node[ver] at (1.59, 1.59){$c$}; % v10 to v9
 \node[ver] at (2.17, 0.58){$d$}; % v9 to v8
 \node[ver] at (2.17, -0.58){$e$}; % v8 to v7
 \node[ver] at (1.59, -1.59){$e'$}; % v7 to v6
 \node[ver] at (0.58, -2.17){$e$}; % v6 to v5
 \node[ver] at (-0.58, -2.17){$d$}; % v5 to v4
 \node[ver] at (-1.59, -1.59){$c$}; % v4 to v3
 \node[ver] at (-2.17, -0.58){$b$}; % v3 to v2
 \node[ver] at (-2.17, 0.58){$a$}; % v2 to v1
 \node[ver] at (-1.59, 1.59){$a'$}; % v1 to v12
 \node[ver] at (-0.58, 2.17){$a$}; % v12 to v11

 \end{tikzpicture}
 \caption{Completing a partial neighborhood of a 12-gon}
 \label{fig:nbdflip12}
\end{figure}

\end{proof}

\begin{definition}
For an ordered tuple $\mathbf{k} = (k_1,\dots,k_n)$, the \emph{associated cyclic tuple} of $\mathbf{k}$ is denoted by $[\mathbf{k}]=[k_1,\dots,k_n]$.

For two ordered tuples $\mathbf{k}_1=(k_1,\dots,k_m)$ and $\mathbf{k}_2=(l_1,\dots,l_n)$, we define the cyclic tuple
\[\mathbf{k}_1 \oplus \mathbf{k}_2:= \big[k_1,\dots,k_m,l_1,\dots,l_n \big].
\]
\end{definition}

The following proposition, though not used in the proof of our main result, provides a heuristic for forming a cyclic-tuple that admits a larger patch of tiling from a given one. 

\begin{proposition}
Let $\mathfrak{k}$ be a cyclic $d$-tuple such that all but one of the faces in $\mathfrak{k}$ admit a full neighborhood. Then there exists an ordered tuple $\bar{\mathbf{k}}$ and a representative $\mathbf{k}$ of $\mathfrak{k}$ such that the exceptional face admits a neighborhood in $\mathbf{k} \oplus \bar{ \mathbf{k}}$, and every other face admits the same neighborhoods for $\mathbf{k} \oplus \bar{ \mathbf{k}}$ as for $\mathfrak{k}$.
\end{proposition}

\begin{proof} 
Suppose $\mathfrak{k}$ contains a $k$-gon in the order $[ \dots, a', a, k, b, b' \dots]$ that does not admit a full neighborhood. This $k$-gon must necessarily be an odd-sided face, and it admits a partial neighborhood of length $d-1$:
 \[
[a, b, a, b, \dots, a]
\]
Consider the representative $\mathbf{k}=( b, b' \dots, a', a, k)$ of $\mathfrak{k}$ and the tuple $\bar{\mathbf{k}}= ( c, k, \underbrace{ b, b', \dots, a, k}_{\mathbf{k}})$, and the cyclic tuple
\[ \mathbf{k} \oplus \bar{ \mathbf{k}}= [ b, b' \dots, a', a, k, c, k, b, b', \dots, a, k]\]

 where $c \notin \mathfrak{k}$. 
 \\
 Then the $k$-gon admits a complete neighborhood of the form 
 \[ [a, b, a, b, \dots, b, c] \]
 
For every other face, the adjacencies are the same in $\mathbf{k} \oplus \bar{ \mathbf{k}}$ as for $\mathfrak{k}$, consequently, the last part of the statement of the proposition holds.

\end{proof}

\section{Cyclic Tuples with Arbitrarily Large Heesch Numbers} \label{Heesch}

The observations made above serve as a guiding principle in the proof of the main theorem (Theorem \ref{heeschn} in \S \ref{intro}). First, we illustrate how these principles allow us to construct an explicit example of a cyclic tuple with Heesch number $1$.

\subsection{Base Example}
\begin{proposition} \label{1butall}
There exists a cyclic tuple that admits a complete $1$-layer patch and a partial $2$-layer patch consisting of full neighborhoods around all but one face in the first corona. In particular, the cyclic tuple has Heesch number $\mathcal H(\mathfrak{k}) = 1$.
	
	\end{proposition}
	\begin{proof}
We consider the cyclic tuple $\tilde{\mathfrak{k}}_1 $, presented with an index set for clarity:

\[
\left[
\begin{array}{cccccccccccccccccc}
	a_1 & a_2 & a_3 & a_4 & a_5 & a_6 & a_7 & a_8 & a_9 & a_{10} & a_{11} & a_{12} & a_{13} & a_{14} \\
	[3pt]
	\hspace{0.3em}8 & 3 & 10 & 3 & 5 & 12 & 5 & 3 &
	8 & 3 & 14 & 3 & 5 & 18 \\
	[6pt]
	 a_{15} & a_{16} & a_{17} &
	a_{18} & a_{19} & a_{20} & a_{21} & a_{22} & a_{23} & a_{24} & a_{25} & a_{26} \\
	[3pt]
	\hspace{0.3em} 
 5 & 3 & 16 & 3 & 10 & 18 & 5 & 20 & 5 & 22 & 5 & 12
\end{array}
\right]
\]

By Lemma \ref{even}, every even-sided face always admits a full neighborhood regardless of how we form the neighborhoods around the odd-sided faces. Therefore, to prove that $\tilde{\mathfrak{k}}_1 $ admits a complete $1$-layer patch $X_1$, it suffices to verify that full neighborhoods can be formed around every sequence of odd-sided faces situated between consecutive even-sided faces in $\mathcal{C}_0=X_0$.

We exhaustively list these odd-sided face sequences in $\mathcal{C}_0$ and their corresponding valid full neighborhoods below:

\begin{table}[H]
 \centering
 \begin{tabular}{|c|c|}
 \hline
 \textbf{Location of triangles in $\tilde{\mathfrak{k}}_1$ with no adjacent pentagon} & \textbf{Neighborhood} \\
 \hline
 $a_2 - (8, 3, 10)$ & $[5, 8, 10]$ \\
 \hline
 $a_{10} -(8, 3, 14)$ & $[5, 8, 14]$ \\
 \hline
 $a_{18} - (16, 3, 10)$ & $[5, 10, 16]$ \\
 \hline
\end{tabular}
 \caption{Isolated triangles and their neighborhoods}
 \label{tab:N3}
\end{table}

\begin{table}[H]
 \centering
 \begin{tabular}{|c|c|}
 \hline
 \textbf{Locations of pentagon in $\tilde{\mathfrak{k}}_1$ with no adjacent triangle} & \textbf{Neighborhood} \\
 \hline
 $ a_{21} - (18, 5, 20), a_{23} - (20, 5, 22), a_{25} - (22, 5, 12)$ & $ [ 3, 18, 20, 22, 12]$ \\
 \hline
 \end{tabular}
 \caption{Neighborhoods of isolated pentagons (see Figure~\ref{fig:nbd35})}
 \label{tab:N5}
\end{table}

\begin{table}[H]
 \centering
\begin{tabular}{|c|c|}
 \hline
 \textbf{Locations of (triangle, pentagon) pair in $\tilde{\mathfrak{k}}_1$} & \textbf{Neighborhoods for the pair} \\
 \hline
 $(a_4, a_5) - (10, 3, 5, 12)$ & $[8, 10, 5], [3, 18, 20, 22, 12]$ \\
 \hline
 $(a_7, a_8) - (12, 5, 3, 8)$ & $[8, 10, 5], [3, 18, 20, 22, 12]$ \\
 \hline
 $(a_{12}, a_{13}) - (14, 3, 5, 18)$ & $[10, 16, 5], [3, 18, 20, 22, 12]$ \\
 \hline
 $(a_{15}, a_{16}) - (18, 5, 3, 16)$ & $[8, 14, 5], [3, 18, 20, 22, 12]$ \\
 \hline
\end{tabular}
\caption{Joint neighborhoods of (triangle, pentagon) pairs (see Figure~\ref{fig:nbd35})}
\label{tab:N35}
\end{table}
 
 		\begin{figure}[H]
 \centering	
		
		\begin{minipage}{0.48\textwidth}
 \centering
 % Tip: Use width=\textwidth instead of scale to make the image perfectly fit the subfigure box
	\begin{tikzpicture}[scale=0.5, transform shape, line cap=round, line join=round]
		% Define inner and outer radii
		\def\rin{4}
		\def\rout{6.5}
		
		% Define styles for lines, text, and the dotted angle arcs
		\tikzset{
			thickline/.style={line width=1 pt},
			innernum/.style={font=\Huge\bfseries\sffamily},
			outernum/.style={font=\LARGE\bfseries\sffamily},
			outernumlarge/.style={font=\fontsize{45}{54}\bfseries\sffamily}, % Extra large for 6
			% Reduced angle radius to bring the dotted arcs closer to the vertices
			anglearc/.style={draw, dash pattern=on 1pt off 2.5pt, line width=1 pt, angle radius=2cm}
		}
		
		% Draw the main inner and outer arcs (open-ended)
		\draw[thickline] (15:\rin) arc (15:165:\rin);
		\draw[thickline] (15:\rout) arc (15:165:\rout);
		
		% 1. Define the 5 corner vertices on the inner arc
		\coordinate (V1) at (75:\rin); % Anchored inside region 5
		\coordinate (V2) at (45:\rin); % Anchored at border of 10 and 5
		\coordinate (V3) at (62:\rin); % Restored vertex inside region 5
		\coordinate (V4) at (95:\rin); % Anchored at border of 5 and 3
		\coordinate (V5) at (130:\rin); % Anchored at border of 3 and 8
		
		% Draw the 3 internal radial lines separating the main numbered regions
		% This makes region 5 (50 degrees) proportionally larger than region 3 (35 degrees)
		\coordinate (Origin) at (0,0);
		\draw[thickline] (Origin) -- (V2); % Border 10/5
		\draw[thickline] (Origin) -- (V4); % Border 5/3
		\draw[thickline] (Origin) -- (V5); % Border 3/8
		
		% 2. Define outer points to create the 5 narrow "V" shapes (approx 10 degrees wide)
		\coordinate (O1R) at (70:\rout); \coordinate (O1L) at (80:\rout);
		\coordinate (O2R) at (40:\rout); \coordinate (O2L) at (50:\rout);
		\coordinate (O3R) at (55:\rout); \coordinate (O3L) at (65:\rout);
		\coordinate (O4R) at (90:\rout); \coordinate (O4L) at (100:\rout);
		\coordinate (O5R) at (125:\rout); \coordinate (O5L) at (135:\rout);
		
		% 3. Draw the "V" shapes radiating to the outer arc
		\draw[thickline] (V1) -- (O1R); \draw[thickline] (V1) -- (O1L);
		\draw[thickline] (V2) -- (O2R); \draw[thickline] (V2) -- (O2L);
		\draw[thickline] (V3) -- (O3R); \draw[thickline] (V3) -- (O3L);
		\draw[thickline] (V4) -- (O4R); \draw[thickline] (V4) -- (O4L);
		\draw[thickline] (V5) -- (O5R); \draw[thickline] (V5) -- (O5L);
		
		% 4. Draw the dotted angle arcs INSIDE the V-shapes
		\pic [anglearc] {angle = O1R--V1--O1L};
		\pic [anglearc] {angle = O2R--V2--O2L};
		\pic [anglearc] {angle = O3R--V3--O3L};
		\pic [anglearc] {angle = O4R--V4--O4L};
		\pic [anglearc] {angle = O5R--V5--O5L};
		
		% 5. Place Inner Labels
		% Centered perfectly within their radial blocks
		\node[innernum] at (30:2.6) {12};
		\node[innernum] at (70:2.6) {5};
		\node[innernum] at (112.5:2.6) {3};
		\node[innernum] at (147.5:2.6) {$j$};
		
		% 6. Place Outer Labels
		% These sit in the wider gaps *between* the V-shapes
		\node[outernum] at (53:5) {22};
		\node[outernum] at (68:5) {20};
		\node[outernum] at (85:5) {18};
		
		% Place the large '6' in the final large gap above region 3
		\node[outernum] at (112.5:5) {$l$};
		
	\end{tikzpicture}
	
 \caption{Neighborhoods of triangle-pentagon pairs for $(j,l)=(8,14),(10,16),(14,16),(16,14)$ }
 \label{fig:nbd35} 
 \end{minipage}
 \hfill
 \begin{minipage}{0.50\textwidth}
 \centering
	\begin{tikzpicture}[scale=0.4, transform shape, nodes={scale=1.5}]
	% Define styles
	\tikzstyle{edge}=[thick]
	\tikzstyle{ver}=[font=\small]
	
	% 1. VERTEX COORDINATES
	% Center of curvature is at (0, -20). 
	
	% Top row (3 vertices for the pentagon) lying on arc of R=20
	\coordinate (v1) at (-3, -0.23);
	\coordinate (v5) at (0, 0);
	\coordinate (v4) at (3, -0.23);
	
	% Middle row lying on arc of R=17.5
	\coordinate (v2) at (-1.5, -2.56);
	\coordinate (v3) at (1.5, -2.56);
	
	% Bottom vertex for the new equilateral triangle
	\coordinate (v0) at (0, -5.16); 
	
	% 2. HORIZONTAL BOUNDARY LINES & LABELS 
	
	% The line containing top of pentagon (\partial X_{n+1}), R=20
	% Spans strictly from x = -5.5 to x = 5.5 (Medium)
	\draw (-5.5, -0.77) arc (105.95:74.05:20) node[right] {$\bf{\partial X_{2}}$};
	
	% The line containing middle vertices (\partial X_n), R=17.5
	% Spans strictly from x = -4.5 to x = 4.5 (Shortest)
	\draw (-4.5, -3.09) arc (104.93:75.07:17.5) node[right] {$\bf{\partial X_1}$};
	
	% 3. POLYGON EDGES
	% Top horizontal edges of pentagon (drawn as arcs to match the boundary)

	% Connecting diagonal sides of pentagon
	\draw[edge] (v1) -- (v2) node[midway, left] {$\textbf{12}$};
	\draw[edge] (v4) -- (v3) node[midway, right] {$\textbf{12}$};
	\node at (-1.5, -3.6) {\textbf{8}};
	\node at (1.5, -3.6) {\textbf{10}};
	
	% Bottom horizontal edge of pentagon (drawn as arc to match boundary)
	\draw[edge] (v2) arc (94.92:85.08:17.5);
	\node at (0, -3.6) {\Huge \textbf{3}};
	
	% Downward Equilateral Triangle edges
	\draw[edge] (v2) -- (v0);
	\draw[edge] (v3) -- (v0);
	
	% 4. SPIKES AND DOTTED CURVES
	% Top 3 spikes (Shorter, and upright relative to the normal of the curved base)
	
	% For v5 (Center, tangent is horizontal, normal is 90 deg)
	\draw (v5) -- ++(115:1.5);
	\draw (v5) -- ++(65:1.5);
	\draw[thick, dotted] (v5) ++(65:0.8) arc (65:115:0.8);

	% For v1 (Left, base is curved, normal is ~98.6 deg)
	% Spread is +/- 25 degrees from the normal
	\draw (v1) -- ++(123.6:1.5);
	\draw (v1) -- ++(73.6:1.5);
	\draw[thick, dotted] (v1) ++(73.6:0.8) arc (73.6:123.6:0.8);

	% For v4 (Right, base is curved, normal is ~81.4 deg)
	% Spread is +/- 25 degrees from the normal
	\draw (v4) -- ++(106.4:1.5);
	\draw (v4) -- ++(56.4:1.5);
	\draw[thick, dotted] (v4) ++(56.4:0.8) arc (56.4:106.4:0.8);

	% Middle 2 spikes (pointing out to \partial X_{n+1})
	\draw (v2) -- (-5, -0.64);
	\draw (v3) -- (5, -0.64);
	
	% Bottom-most spikes connecting to \partial X_n
	\draw (v0) -- (-3.5, -2.85);
	\draw (v0) -- (3.5, -2.85);
	
	% 5. CENTRAL NUMBER
	\node at (0, -1.3) {\Huge \textbf{5}};
\end{tikzpicture}
 \caption{Forced blocking partial neighborhood of a $5$-gon in the second layer}
 \label{around3}
 \end{minipage}

\end{figure}

This exhaustively verifies all possible sequences of odd-sided faces between any two even-sided faces. Consequently, a complete $1$-layer ($X_1$) patch of type $\tilde{\mathfrak{k}}_1$ exists.

Next, we analyze the first corona $\mathcal{C}_1$. Note that the unique neighborhood that the triangle at position $a_2$ in $\mathcal{C}_0$ possesses forces the specific partial neighborhood $(12, 3, 12)$ onto the adjacent pentagon in $\mathcal{C}_1$; see Figure~\ref{around3}. Due to the parity constraint, this specific partial neighborhood cannot be extended into a full neighborhood of the pentagon. Therefore, $\tilde{\mathfrak{k}}_1$ does not admit a complete second corona $\mathcal{C}_2$; consequently, 

 \[\mathcal{H}(\tilde{\mathfrak{k}}_1)=1.\]

To achieve the stronger property in the proposition---a partial $2$-layer patch where \textit{all but one} face in the first corona possesses a full neighborhood---we employ a similar strategy: we prioritize forming neighborhoods around odd-sided sequences in $\mathcal{C}_1$ before addressing the even-sided faces. For the even-sided faces in $\mathcal{C}_0$, it is convenient to adapt the methodology provided in Lemma \ref{even} (b), which was originally applied for even-sided faces in $\mathcal{C}_i$ with $i \geq 1$ though. To facilitate this, we construct a modified cyclic tuple $\mathfrak{k}_1$ from $\tilde{\mathfrak{k}}_1$ by replacing each isolated even-sided face with two identical copies and increasing their respective sizes by $4$. This modification yields the following cyclic tuple:

\begin{equation} \label{k1}
\mathfrak{k}_1=A_1\oplus B_1\oplus C_1\oplus D_1,
\end{equation}
where
\[
\begin{aligned}
A_1&=(12,3,14^{\times 2},3,5,16^{\times 2},5,3,12), \quad B_1 =(12,3,18^{\times 2},3,5,22),\\
C_1&=(14,3,20^{\times 2},3,5,22), \quad D_1=(22,5,24^{\times 2},5,26^{\times 2},5,16).
\end{aligned}
\]

 \textbf{Construction of $\mathcal{C}_1$:} Observe that the number of the odd-sided faces and their immediate adjacencies remain the same in $\mathfrak{k}_1$ as it was in $\tilde{\mathfrak{k}}_1$. Consequently, the neighborhoods for any sequence of odd-sided faces bounded by even-sided faces in $\mathcal{C}_0$ can be constructed using exactly the same neighborhoods for $\tilde{\mathfrak{k}}_1$. 
\\
Furthermore, no odd-even-odd subtuple occurs in $\mathfrak{k}_1$. This allows us to form the neighborhoods around the even-sided faces in $\mathcal{C}_0$ by the reversal method of Lemma~\ref{even}(b), more precisely by its partial two-layer version.
 
\textbf{Construction of an incomplete (all but one) $\mathcal{C}_2$:} We now proceed to construct the neighborhoods for the faces in the first corona $\mathcal{C}_1$, addressing the odd-sided faces first. The neighborhoods for odd-sided faces in $\mathcal{C}_1$ that are adjacent to even-sided faces in $\mathcal{C}_0$ have already been determined. For odd-sided faces in $\mathcal{C}_1$ that are adjacent to odd-sided faces in $\mathcal{C}_0$, we have two cases: 
\\
\textbf{Case A.} For isolated odd-sided faces in $\mathcal{C}_1$ that share a single vertex with the boundary $\partial \mathcal{C}_1$, its full neighborhoods can be constructed using the same neighborhoods used in $0$-th corona. The same applies to a pair of odd-sided faces each sharing single vertices with the boundary $\partial \mathcal{C}_1$. Because odd-sided faces in the cyclic tuple appear either isolated or in pairs, an odd-sided face that shares an edge with $\partial \mathcal{C}_1$ must either be isolated or belong to a pair of odd-sided faces. 
\\
\textbf{Case B.} For odd-sided faces (either isolated or in pairs) that share single edge with $\partial \mathcal{C}_1$, we first consider the pentagon in $\mathcal{C}_2$ that appears in the neighborhood of an isolated triangle in $\mathcal{C}_0$. This occurs for triangles in the positions $a_2$, $a_{10}$ and $a_{16}$ (see Table \ref{tab:N3}). The first case of $a_2$ creates the blocking we have already discussed above. In both the remaining cases, we form the neighborhood $[3, 12, 18, 20, 22]$ around the pentagon which is identical to the one formed in $\mathcal{C}_0$, as shown in Fig. \ref{fig:nbd35}. An analogous construction works for triangles in the neighborhood of the isolated pentagons in $\mathcal{C}_0$ at positions $a_{21}$, $a_{23}$ and $a_{25}$ (see Table \ref{tab:N5}). By our construction, the edge-adjacent neighbors of a pair of odd-sided faces in $\mathcal{C}_1$ are all even-sided faces. This completes the neighborhood construction of the neighborhoods of the odd-sided faces in $\mathcal{C}_1$.
\\
For the even-sided faces in $\mathcal{C}_1$, we can appeal again to the first part of the Lemma \ref{even}. 
\\
Thus all the faces in $\mathcal{C}_1$ admit a joint neighborhood with the sole exception of the unique pentagon attached to the triangle in $\mathcal{C}_0$ at position $a_2$. 
This completes the proof.

\end{proof}

\subsection{Proof of Theorem~\ref{heeschn}}

	To construct a cyclic tuple with a Heesch number of $2$, we strategically extend the cyclic tuple $\mathfrak{k}_1$. This extension is designed to resolve the combinatorial obstruction surrounding the pentagon in the second layer while simultaneously introducing a new, strictly forced obstruction around a different odd-sided face in the third layer. By iterating this localized blocking mechanism, we can generate cyclic tuples possessing arbitrarily high Heesch numbers.

\begin{proof}

For a finite sequence of integers $a_1,\dots,a_m$ and an integer $\alpha$, let
$$ \mathcal{J}_\alpha(a_1,\dots,a_m) :=\left(a_1,\alpha,(a_2)^{\times 2},\alpha,\dots,\alpha,(a_{m-1})^{\times 2}, \alpha, a_m\right). $$

For $i \geq 1$, define $$\tilde{\mathfrak{k}}_i = A_i \oplus B_i \oplus C_i \oplus D_i,$$ where

 \[
\begin{aligned} 
A_i &:= \left(k^i_1, 2i+1, (k^i_2)^{\times 2}, 2i+1, 2i+3, (k^i_3)^{\times 2}, 2i+3, 2i+1, k^i_1\right),\\ 
B_i &:= \left(\mathcal{J}_{2i+1}(k^i_1, k^i_5, k^i_6,\dots,k^i_{2i+3}, 2i+3), k^i_4\right)\\ 
&= (k^i_1, 2i+1, (k^i_5)^{\times 2}, 2i+1, \dots, 2i+1, (k^i_{2i+3})^{\times 2}, 2i+1, 2i+3, k^i_{4}),\\ 
C_i &:= \left(\mathcal{J}_{2i+1}(k^i_2, k^i_{2i+4},\dots,k^i_{4i+2}, 2i+3), k^i_4\right)\\
&=(k^i_2, 2i+1, (k^i_{2i+4})^{\times 2}, 2i+1, \dots, 2i+1, (k^i_{4i+2})^{\times 2}, 2i+1, 2i+3, k^i_4), \\
D_i &:= \mathcal{J}_{2i+3}(k^i_4, k^i_{4i+3},\dots,k^i_{6i+2}, k^i_3)\\
&=(k^i_4, 2i+3, (k^i_{4i+3})^{\times 2}, 2i+3, \dots, 2i+3, (k^i_{6i+2})^{\times 2}, 2i+3, k^i_3).
\end{aligned} 
\]

Here, $k^i_j \geq 12$ are even integers for all $j \geq 1$,  pairwise distinct for all $i, j$ except for the overlap condition: 
\[
 k^i_3= k^{i+1}_1 \quad \text{for } i \ge 1.
\]
Such a choice of integers is always possible.

For $n \ge 2$, we define the cyclic tuple
\[
\mathfrak{k}_n = \bigoplus_{i=1}^n \tilde{\mathfrak{k}}_i.
\]

We claim that $\mathfrak{k}_n$ has a Heesch number of $n$. Observe that for the base case $i=1$, the following assignment
\[ (k_1^1,\ldots,k_8^1)=(12,14,16,22,18,20,24,26) \]

yields precisely the cyclic tuple $\mathfrak{k}_1$ described in Eq. \ref{k1} in Proposition \ref{1butall} which we have already established has Heesch number $1$. Hence, the n=1 case of the theorem follows from Proposition \ref{1butall}.

 We build the successive coronas by iterating the local extension rules and verify the construction by induction on the corona index. During the construction process, the partial neighbourhoods induced on a new outer corona have two different behaviours. Most belong to recurrent local
configurations: whenever one of these configurations occurs, it can be completed by one of the same neighbourhood constructions used in the
preceding coronas.  In addition, there is one distinguished sequence of partial neighbourhoods of certain odd-sided faces that changes from one corona to the next.  The $(r+1)$-th corona is formed by first prioritizing the odd-sided faces in $r$-th corona, followed by the even-sided faces.

\textbf{Neighborhood formation of odd faces.} Note that the odd-sided faces situated between two consecutive even-sided faces in $\mathfrak{k}_n$ must appear either isolated or in pairs of the form $((2i+1), (2i+3))$. We systematically divide the neighborhood construction around the odd faces according to the nature of their intersecttion with the boundary of the preceding corona.

\textbf{Case A.} We first address the scenario in which the odd-sided faces in $X_k$--whether isolated or appearing in pairs--share only single vertices with the boundary of the preceding corona $\partial \mathcal{C}_{k-1}$. This condition holds inherently for the zeroth corona. For subsequent coronas, the neighborhoods can be constructed identically to those in the zeroth corona. All possible positions of the odd-sided faces in any corona in the iterative construction and the corresponding valid neighborhood formations are detailed below:

\begin{enumerate}
\item
Let $\arg_X(\mathcal{J}_\alpha)$ denote the argument passed to the operator $\mathcal{J}_\alpha$ within the definition of block $X$.

\begin{enumerate}

	\item For the isolated $(2i+1)$-gon in the second position of the block $A_i$, 
 the neighborhood is

 \[\mathcal{N}_{A_i}(2i+1)= [ k^i_1, k^i_2, k^i_1, k^i_2, \dots, k^i_1, 2i+3, k^i_2 ].\]

\item
For $(2i+1)$-gons in block $B_i$, the unique neighborhood is
 \[\mathcal{N}_{B_i}(2i+1)= [\arg_{B_i}(\mathcal{J}_{2i+1})] = [k^i_1, k^i_5, \dots, k^i_{2i+3}, 2i+3] .\]
\item
For $(2i+1)$-gons in block $C_i$, the unique neighborhood is
 \[\mathcal{N}_{C_i}(2i+1)= [ \arg_{C_i}(\mathcal{J}_{2i+1}) ]= [k^i_2, k^i_{2i+4},\dots,k^i_{4i+2},2i+3] .\]

\end{enumerate}

\item
For the $(2i+3)$-gons appearing in block $D_i$, the unique neighborhood is
\[\mathcal{N}_{D_i}(2i+3) =[(2i+1), \arg_{D_i}(\mathcal{J}_{2i+1})] = [2i+1, k^i_4, k^i_{4i+3},\dots,k^i_{6i+2}, k^i_3 ].\]
\item
For the pair $\big((2i+1),(2i+3)\big)$ appearing in the block $A_i$ as $\big(\dots,(k^i_2)^{\times 2},(2i+1),(2i+3),(k^i_3)^{\times 2},\dots\big)$, the unique joint neighborhood is
\[
\mathcal{N}_{A_i}\big((2i+1),(2i+3)\big)
= \left(\mathcal{N}_{B_i}(2i+1),\ \mathcal{N}_{D_i}(2i+3) \right),
\]
as illustrated in Fig~\ref{fig:nbd57} below for $i=2$.
\item
Similarly, for the pair $\big ((2i+3), (2i+1) \big)$ appearing in the block $A_i$ as $\big( \dots (k^i_3)^{\times 2},(2i+3),(2i+1),(k^i_1)^{\times 2}, \dots \big)$, the unique joint neighborhood is \[ 
\mathcal{N}_{A_i} \big ((2i+3), (2i+1) \big) = \left( \mathcal{N}_{D_i} (2i+3), \mathcal{N}_{C_i}(2i+1) \right ).
 \]
\end{enumerate}

\begin{figure}[H]
 \centering	
		
		\begin{minipage}{0.48\textwidth}
 \centering
\resizebox{\textwidth}{!}{
\begin{tikzpicture}[line cap=round, line join=round]
 % Define inner and outer radii
 \def\rin{4}
 \def\rout{6}
 
 % Define styles for lines, text, and the dotted angle arcs
 \tikzset{
 thickline/.style={line width=1.5 pt},
 innernum/.style={font=\Huge\bfseries\sffamily},
 outernum/.style={font=\LARGE\bfseries\sffamily},
 outernumlarge/.style={font=\fontsize{45}{54}\bfseries\sffamily},
 % You may want to slightly increase angle radius if you make the wedges much wider
 anglearc/.style={draw, dash pattern=on 0pt off 4pt, line cap=round, line width=1.5pt, angle radius=1.5cm}
 }
 
 % Draw the main inner and outer arcs (open-ended)
 \draw[thickline] (15:\rin) arc (15:165:\rin);
 \draw[thickline] (15:\rout) arc (15:165:\rout);
 
 % 1. Define the corner vertices on the inner arc for the 7-gon and 5-gon
 \coordinate (V1) at (45:\rin); 
 \coordinate (V2) at (56:\rin); 
 \coordinate (V3) at (67:\rin); 
 \coordinate (V4) at (78:\rin); 
 \coordinate (V5) at (89:\rin); 
 \coordinate (V6) at (100:\rin); 
 
 \coordinate (V7) at (113.3:\rin); 
 \coordinate (V8) at (126.6:\rin); 
 \coordinate (V9) at (140:\rin); 
 
 % Draw the internal radial lines separating the main numbered regions (10 | 7 | 5 | 8)
 \coordinate (Origin) at (0,0);
 \draw[thickline] (Origin) -- (V1); 
 \draw[thickline] (Origin) -- (V6); 
 \draw[thickline] (Origin) -- (V9); 
 
 % 2. INCREASED SPREAD: Outer points set to +/- 3 degrees from their anchors (6 degrees wide)
 \coordinate (O1R) at (42:\rout); \coordinate (O1L) at (48:\rout); % Anchor: 45
 \coordinate (O2R) at (53:\rout); \coordinate (O2L) at (59:\rout); % Anchor: 56
 \coordinate (O3R) at (64:\rout); \coordinate (O3L) at (70:\rout); % Anchor: 67
 \coordinate (O4R) at (75:\rout); \coordinate (O4L) at (81:\rout); % Anchor: 78
 \coordinate (O5R) at (86:\rout); \coordinate (O5L) at (92:\rout); % Anchor: 89
 \coordinate (O6R) at (97:\rout); \coordinate (O6L) at (103:\rout); % Anchor: 100
 \coordinate (O7R) at (110.3:\rout); \coordinate (O7L) at (116.3:\rout); % Anchor: 113.3
 \coordinate (O8R) at (123.6:\rout); \coordinate (O8L) at (129.6:\rout); % Anchor: 126.6
 \coordinate (O9R) at (137:\rout); \coordinate (O9L) at (143:\rout); % Anchor: 140
 
 % 3. Draw the "V" shapes radiating to the outer arc
 \draw[thickline] (V1) -- (O1R); \draw[thickline] (V1) -- (O1L);
 \draw[thickline] (V2) -- (O2R); \draw[thickline] (V2) -- (O2L);
 \draw[thickline] (V3) -- (O3R); \draw[thickline] (V3) -- (O3L);
 \draw[thickline] (V4) -- (O4R); \draw[thickline] (V4) -- (O4L);
 \draw[thickline] (V5) -- (O5R); \draw[thickline] (V5) -- (O5L);
 \draw[thickline] (V6) -- (O6R); \draw[thickline] (V6) -- (O6L);
 \draw[thickline] (V7) -- (O7R); \draw[thickline] (V7) -- (O7L);
 \draw[thickline] (V8) -- (O8R); \draw[thickline] (V8) -- (O8L);
 \draw[thickline] (V9) -- (O9R); \draw[thickline] (V9) -- (O9L);
 
 % 4. Draw the dotted angle arcs INSIDE the V-shapes
 \pic [anglearc] {angle = O1R--V1--O1L};
 \pic [anglearc] {angle = O2R--V2--O2L};
 \pic [anglearc] {angle = O3R--V3--O3L};
 \pic [anglearc] {angle = O4R--V4--O4L};
 \pic [anglearc] {angle = O5R--V5--O5L};
 \pic [anglearc] {angle = O6R--V6--O6L};
 \pic [anglearc] {angle = O7R--V7--O7L};
 \pic [anglearc] {angle = O8R--V8--O8L};
 \pic [anglearc] {angle = O9R--V9--O9L};
 
 % 5. Place Inner Labels
 \node[innernum] at (30:2.6) {$k^{\scriptscriptstyle 2}_{_3}$};
 \node[innernum] at (72.5:2.6) {7};
 \node[innernum] at (120:2.6) {5};
 \node[innernum] at (154.5:2.6) {$k^{\scriptscriptstyle 2}_{_1}$};
 
 % 6. Place Outer Labels
 \node[outernum] at (50.5:5) {\small $k^2_{14}$};
 \node[outernum] at (61.5:5) {\small $k^2_{13}$};
 \node[outernum] at (72.5:5) {\small $k^2_{12}$};
 \node[outernum] at (83.5:5) {\small $k^2_{11}$};
 \node[outernum] at (94.5:5) {\small $k^2_4$};
 \node[outernum] at (106.6:5) {\small $k^2_7$};
 \node[outernum] at (119.9:5) {\small $k^2_6$};
 
 % Place the large '6'
 \node[outernum] at (133.3:5) {\small $k^2_5$ };
 
\end{tikzpicture}}

 \caption{Neighborhood $\mathcal{N}_{B_i}(5,7)$}
 \label{fig:nbd57} 
 \end{minipage}

\end{figure}

\textbf{Case B.}
Next, we consider the cases in which an odd-sided face in the corona $\mathcal{C}_k$ shares an edge with $\partial \mathcal{C}_{k-1}$. 
\begin{description}	
\item[Subcase B.1.] Consider an odd-sided face in $\mathcal{C}_k$ that appears within the neighborhood of an odd-sided face in $\mathcal{C}_{k-1}$, which itself shares only a single vertex with $\partial \mathcal{C}_{k-2}$. 
\\
Such configurations arise within the neighborhoods detailed in Case A.1(a)-(b)-(c): $\mathcal{N}_{A_i}(2i+1)$, $\mathcal{N}_{B_i}(2i+1)$, $\mathcal{N}_{C_i}(2i+1)$, and Case A.2: $\mathcal{N}_{D_i}(2i+3)$. With the sole exception of $i=1$ in Case A.1(a), the partial neighborhoods around these odd-sided faces can be reliably extended to a full neighborhoods using the neighborhood described above.

In the exceptional case--where a $5$-gon ($2i+3$ for $i=1$) appears in the neighborhood $\mathcal{N}_{A_1}(3)$ of a triangle--the partial neighborhood $F_1= (k^2_1, 3, k^2_1)$ uniquely extends to the following full neighborhood:
			\[ \mathcal{N}_{F_1}(5) = [ 3, k^2_1, k^2_2, 7, k^2_1 ]. \]
This specific configuration is illustrated in Fig.~\ref{around7}.

\item[Subcase B.2.] 
Consider an odd-sided face in $\mathcal{C}_k$ that shares an edge with an odd-sided face in $\mathcal{C}_{k-1}$, which in turn shares an edge with $\partial \mathcal{C}_{k-2}$. Such an instance occurs, for example, for the $7$-gon in the neighborhood of type $\mathcal{N}_{F_1} (5)$ formed in Subcase B.1.

By the design of the $A_i$ block in $\mathfrak{k}_n$, the ordered tuple $(k^{i}_3, 2i+1, k^{i}_3)$, denoted by $F_i$, represents a valid partial neighborhood around a $(2i+3)$-gon. By the fundamental overlap condition $k^{i}_3= k^{i+1}_1$ in $\mathfrak{k}_n$ for $i \geq 1$, 
\[ F_i = (k^{i+1}_1, 2i+1, k^{i+1}_1) \]

We can extend the partial neighborhood $F_i$ around $2i+3$-gon to the neighborhood
\begin{equation}
\mathcal{N}_{F_i}(2i+3) = [ 2i+1, (k^{i+1}_1, k^{i+1}_2)^{\times r}, 2i+5, (k^{i+1}_1, k^{i+1}_2)^{\times s}, k^{i+1}_1 ] \quad \text{for} \ 2 \leq i \leq n-1,
\end{equation}
where $r > 1$ and $s \geq 0$ and $r+s=i$.

We claim that for each $2 \leq i \leq n-1$, an instance of Subcase B.2. occurs for a $(2i+3)$-gon in some corona $\mathcal{C}_k$ in the neighborhood of type $\mathcal{N}_{F_{i-1}}(2i+1)$ of a $(2i+1)$-gon in the previous corona $\mathcal{C}_{k-1}$, and we form a neighborhood of type $\mathcal{N}_{F_i}(2i+3)$ around such $2i+3$-gon.

The base case of the finite induction corresponds to $i=2$ for the $7$-gon in the neighborhood of type $\mathcal{N}_{F_1} (5)$ of a $5$-gon, as illustrated in Fig.~\ref{around7}. 
Observe that, by the design of the block $A_{i+1}$ in $\mathfrak{k}_n$, the segment $(k^{i+1}_2, 2i+5, k^{i+1}_1)$ within $\mathcal{N}_{F_i}(2i+3)$ forces the partial neighborhood $(k^{i+1}_3, 2i+3, k^{i+1}_3)$ around the $(2i+5)$-gon (in $\mathcal{N}_{F_i}(2i+3)$) in the next corona. Again by the overlap condition in $\mathfrak{k}_n$, $(k^{i+1}_3, 2i+3, k^{i+1}_3) = F_{i+1}$. Hence the partial neighborhood $F_{i+1}$ can be extended into the full neighborhood $\mathcal{N}_{F_{i+1}}(2i+5)$ for $i \leq n-1$, completing the induction.

		\begin{figure}[H]
		\centering
		\begin{tikzpicture}[scale=0.6, transform shape]
	% Define styles 
	\tikzstyle{edge}=[thick]
	\tikzstyle{ver}=[font=\small]
	
	% 1. VERTEX COORDINATES
	% Center of curvature remains at (0, -20). 
	
	% Top row (3 vertices for the pentagon) lying on arc of R=20
	\coordinate (v1) at (-3, -0.23);
	\coordinate (v5) at (0, 0);
	\coordinate (v4) at (3, -0.23);
	
	% Middle row lying on arc of R=17.5
	\coordinate (v2) at (-1.5, -2.56);
	\coordinate (v3) at (1.5, -2.56);
	
	% Bottom vertex for the new equilateral triangle
	\coordinate (v0) at (0, -5.16); 
	
	% 2. HORIZONTAL BOUNDARY LINES & LABELS (Tapered Lengths)
	
	% Topmost boundary line (\partial X_{n+2}), R=22.5
	% Spans strictly from x = -6.5 to x = 6.5 (Longest)
	\draw (-6.5, 1.54) arc (106.80:73.20:22.5) node[right] {\Large \textbf{$\partial X_{n+2}$}};
	
	% The line containing top of pentagon (\partial X_{n+1}), R=20
	% Spans strictly from x = -5.5 to x = 5.5 (Medium)
	\draw (-5.5, -0.77) arc (105.95:74.05:20) node[right] {\Large \textbf{$\partial X_{n+1}$}};
	
	% The line containing middle vertices (\partial X_n), R=17.5
	% Spans strictly from x = -4.5 to x = 4.5 (Shortest)
	\draw (-4.5, -3.09) arc (104.93:75.07:17.5) node[right] {\Large \textbf{$\partial X_{n}$}};
	
	% 3. POLYGON EDGES
	% Top horizontal edges of pentagon (drawn as arcs to match the boundary)
	\draw[edge] (v1) arc (98.63:90:20) node[midway, above] {\large $k^2_2$};
	\draw[edge] (v5) arc (90:81.37:20) node[midway, above] {\Huge \textbf{7}};
	
	% Connecting diagonal sides of pentagon
	\draw[edge] (v1) -- (v2) node[pos=0.3, left=0.15cm] {\large$k^1_3=k^2_1$};
	\draw[edge] (v4) -- (v3) node[pos=0.25, right=0.1cm] {\large $k^1_3=k^2_1$};
	\node at (-1.5, -3.6) { \large $k^1_1$};
	\node at (1.5, -3.6) { \large $k^1_2$};
	
	% Bottom horizontal edge of pentagon (drawn as arc to match boundary)
	\draw[edge] (v2) arc (94.92:85.08:17.5);
	\node at (0, -3.6) {\Huge \textbf{3}};
	
	% Downward Equilateral Triangle edges
	\draw[edge] (v2) -- (v0);
	\draw[edge] (v3) -- (v0);
	
	% Diagonal edges going up to the top boundary
	\draw[edge] (v4) -- (3, 2.3);
	\draw[edge] (v5) -- (1.8, 2.42);
	
	% 4. SPIKES AND DOTTED CURVES
	% Top 3 spikes (Centers pointing UPWARD to \partial X_{n+2})
	% For \v = v1
	\draw (v1) -- (-3.8, 2.17);
	\draw (v1) -- (-2.2, 2.39);
	\draw[thick, dotted] (v1) ++(-0.4, 1.25) to[bend left=30] ++(0.8, 0);

	% For \v = v5
	\draw (v5) -- (-0.8, 2.48);
	\draw (v5) -- (0.8, 2.48);
	\draw[thick, dotted] (v5) ++(-0.4, 1.25) to[bend left=30] ++(0.8, 0);

	% For \v = v4
	\draw (v4) -- (3.8, 2.17);
	\draw (v4) -- (5.2, 1.89);
	\draw[thick, dotted] (v4) ++(0.5, 1.4) to[bend left=30] ++(0.7, -0.1);

	% Middle 2 spikes (Centers pointing UPWARD to \partial X_{n+1})
	% For left side
	\draw (v2) -- (-5, -0.64);
		
	% For right side
	\draw (v3) -- (5, -0.64);
	
	% Bottom-most spikes connecting to \partial X_n
	\draw (v0) -- (-3.5, -2.85);
	\draw (v0) -- (3.5, -2.85);
	
	% 5. CENTRAL NUMBER
	\node at (0, -1.3) {\Huge \textbf{5}};
	
	% Labels around 7 gon
	\node at (0.75, 1.5) {$k^2_3$};
	\node at (1.0, 2) {$k^3_1{=}$};
	\node at (3.3, 1.55) {$k^2_3$};
	\node at (3.4, 2.0) {$k^3_1{=}$};
		
\end{tikzpicture}

 \caption{Propagation of forced partial neighborhoods}
 \label{around7}	
\end{figure}

	\item[Subcase B.3.] Finally, consider the cases where odd-sided faces in the corona $\mathcal{C}_k$ share an edge with an even-sided face in $X_{k-1}$. Neighborhood of these faces are constructed using the partial neighborhood reversal method described in Lemma~\ref{even}(b), and they inherit the corresponding neighborhood structures from the lower coronas.
		\end{description}

\textbf{Neighborhood formation of even faces.} After the odd-sided faces have been completed, the remaining
even-sided faces are treated by Lemma~\ref{even} by the systematic reversal and reflection technique.  Part~(a) provides
their full neighbourhoods, while part~(b), using the absence of an
odd-even-odd subtuple, allows these neighbourhoods to be chosen so that no additional type of partial neighborhoods are 
are formed around the odd-sided face.

Thus the above partial neighbourhood completion process around odd and even-sized faces in a $r$-th corona 
produces either the same type of partial neighborhood around the odd-sized faces in the $(r+1)$-th corona treated in Cases~A, ~B.1 and~B.3,
or make the transition 
\[
F_i\longmapsto F_{i+1},\qquad 1\leq i<n.
\]
treated in Case~B.2. The later process constructs the successive coronas through
\(\mathcal C_n\); since \(F_n\) lies in the outermost corona and need
not be completed to obtain \(X_n\), the tuple \(\mathfrak{k}_n\)
admits an \(n\)-layer patch.

\textbf{Obstruction at the $(n+1)$-th layer:} 
We have already observed that the unique neighborhood $\mathcal{N}_{A_1}(3)$ of the triangle located in $A_1$ block at the zeroth corona forces a partial neighborhood of the form $F_1=(k^2_1, 3, k^2_1)$ around the pentagon in the $1$st corona. 

For $1 \leq i \leq n-1$, parity constraints for odd-sided faces imply that any full neighborhood obtained by extending the partial neighborhood $F_i=(k^{i+1}_1, 2i+1, k^{i+1}_1)$ around the $(2i+3)$-gon in some corona (say $\mathcal{C}_{k}$) must include a segement (partial neighborhood) of the form $(k^{i+1}_1, 2i+5, k^{i+1}_2)$. This was observed in Subcase B.2 for the particular choice of neighborhood of type $\mathcal{N}_{F_i}(2i+3)$.

The  partial neighborhood $(k^{i+1}_1, 2i+5, k^{i+1}_2)$, in turn, forces a partial neighborhood of the form $(k^{i+1}_3, 2i+3, k^{i+1}_3)$ around the $(2i+5)$-gon in the next corona $\mathcal{C}_{k+1}$. By the overlap condition $k^i_3= k^{i+1}_1$ in $\mathfrak{k}_n$, this partial neighborhood is precisely $F_{i+1}$ around $(2(i+1)+3)$-gon. Thus, inductively, we have a continuous chain of forced combinatorial dependencies cascading outward in subsequent layers. 

Initiating this chain from the partial neighborhood $F_1$ around the pentagon in the first corona, we ultimately have a forced partial neighborhood $F_n=(k^{n}_3, 2n+1, k^{n}_3)$ around a $(2n+3)$-gon at the $n$-th corona $\mathcal{C}_n$. However, the cyclic tuple $\mathfrak{k}_n$ inherently does not contain a $(2n+5)$-gon. Therefore, $F_n$ cannot be extended into a valid full neighborhood of the $(2n+3)$-gon, proving that $\mathfrak{k}_{n}$ does not admit a complete $(n+1)$-th corona. Thus,
\[ \mathcal{H}(\mathfrak{k}_n)= n\]

With the above block definitions, we verify that
\[
|\tilde{\mathfrak{k}}_i| = 18i + 16, \qquad |\mathfrak{k}_n| = \sum_{i=1}^{n} (18i + 16) = 9n^2 + 25n.
\]

This completes the proof of the theorem.
\end{proof}

\begin{remark} \label{allbutone}(all-but-one property) A cyclic tuple $\mathfrak{k}$ with $\mathcal H(\mathfrak{k})=r$ has the \emph{all-but-one extension property at level $r$} if it admits an $r$-layer patch together with a compatible partial $(r+1)$-st corona that simultaneously completes the neighbourhoods of all but one face in $\mathcal C_r$. It is evident from the construction in the above theorem that cyclic tuples $\mathfrak{k}_n$ has all-but-one extension property at level $n$ 

\end{remark}

\subsection{Passage to Convex Monotiles}\label{convexh}

Dualizing an $n$-layer homogeneous patch of type $\mathfrak{k}_n$ produces an edge-to-edge $n$-layer patch by a convex monotile, but it does not provide a corresponding upper bound for the Heesch number of that tile. Indeed, $\mathfrak{k}_n$ contains even types: a corner corresponding to an even type $e$ has angle $2\pi/e$, and $e/2$ such corners have total angle $\pi$. The angle condition therefore does not exclude completed $T$-junctions in non-edge-to-edge patches; even in the edge-to-edge setting, mixed corner types may satisfy the full angle equation. Consequently, patches by the canonical dual need not be duals of homogeneous patches of type $\mathfrak{k}_n$.
\\
To remove this ambiguity, we construct from the same obstruction-propagation scheme a larger all-prime tuple $\mathfrak{p}_n$ with $\mathcal H(\mathfrak{p}_n)=n$. The primality of its entries provides the rigidity needed to control arbitrary patches by its canonical dual. We first present the following elementary arithmetic fact used for this purpose.

	\begin{lemma} \label{reciprocal1}
		
		Let $p_1,\dots,p_m$ be distinct primes and let $a_1,\dots,a_m$ be positive integers.
		Then
		\[
		\sum_{i=1}^m \frac{a_i}{p_i}=1
		\]
		if and only if $m=1$ and $a_1=p_1$. 
	\end{lemma}
	
	\begin{proof}
		Suppose
		\[
		\sum_{i=1}^m \frac{a_i}{p_i}=1
		\]
		holds with $m\ge1$ and $a_i\in\mathbb{Z}_{>0}$. Setting 
	$ Q:=\prod_{i=1}^m p_i$, and multiplying the above equation by $Q$ gives 
		\[
		\sum_{i=1}^m a_i\frac{Q}{p_i}=Q. \tag{1}
		\]
		
		Fix an index $j\in\{1,\dots,m\}$. For $i\neq j$, the factor $Q/p_i$ is divisible by $p_j$, hence
		\[
		a_i\frac{Q}{p_i}\equiv 0\pmod{p_j}\quad\text{for }i\neq j.
		\]
		Reducing (1) modulo $p_j$ therefore yields
		\[
		a_j\frac{Q}{p_j}\equiv Q\equiv 0\pmod{p_j}.
		\]
		Since $\gcd(Q/p_j,p_j)=1$, multiplication by the inverse of $Q/p_j$ modulo $p_j$ shows that $p_j\mid a_j$. Thus we may write
		\[
		a_j = p_j t_j \quad\text{with } t_j\in\mathbb{Z}_{\ge1},
		\]
		for every $j=1,\dots,m$.
		
		Substituting $a_j=p_j t_j$ into the original sum gives
		\[
		\sum_{j=1}^m \frac{a_j}{p_j}=\sum_{j=1}^m t_j = 1.
		\]
		But each $t_j\ge1$, so the only way their sum can equal $1$ is that $m=1$ and $t_1=1$. Therefore $a_1=p_1$ and no other primes occur.
		
		Conversely, if $m=1$ and $a_1=p_1$ then $\dfrac{a_1}{p_1}=1$. This completes the proof.
	\end{proof}

	\begin{theorem}\label{convexthm} For any given positive integer $n$, there exists a convex monotile $T_n$ with Heesch number $n$.
	\end{theorem}

	\begin{proof} Let $\mathfrak{k}_{n}$ be the cyclic tuple with Heesch number $n$ constructed in Theorem \ref{heeschn}. We construct a cyclic tuple $ \mathfrak{p}_n^{\circ}$ by modifying $\mathfrak{k}_{n}$ as follows: \\
	
(i) Choose consecutive $n+1$ odd primes
\[
p_1=3,\qquad p_2=5,\qquad p_3,\ldots,p_{n+1}.
\]
For each \(1\leq i\leq n\), replace the odd types \(2i+1\) and
\(2i+3\) in the \(i\)-th block by \(p_i\) and \(p_{i+1}\),
respectively. The argument lists in the $B_i$- and $C_i$-blocks are lengthened so as to contain $p_i$ entries, and similarly the argument list in the $D_i$-block is lengthened to contain $p_{i+1}-1$ entries. 

(ii) Replace each distinct even integer by primes \(q_1, q_2, \dots q_m \), where the
\(q_i\)'s are distinct odd primes, disjoint from the primes \(p_i\)'s. Repeated occurrences
of the same even integers, including the prescribed overlaps, are
replaced by the same prime.

 Then we define an auxilary tuple \[E_n=
\bigoplus_{1\le i,j\le m}(q_1,q_i,q_j)
\oplus
\bigoplus_{i=1}^m(q_1,q_i,q_1).\]
 Our final cyclic tuple is
 \[ \mathfrak{p}_n= \mathfrak{p}_n^{\circ} \oplus E_n.\]
 where $E_n$ is inserted in a $q-q$ cut of $\mathfrak{p}_n^{\circ}$

 The incidences involving the $p_i$'s are unchanged by the
modification of $\mathfrak{k}_n$, and $E_n$ contains no $p_i$. Hence the blocking argument
of Theorem~\ref{heeschn} carries over unchanged and gives
\[
\mathcal H(\mathfrak{p}_n)\leq n.
\]
Conversely, an $n$-layer tiling patch of type \(\mathfrak{p}_n\) can be constructed by following the construction for $\mathfrak{k}_n$. Firstly, any partial neighborhood of the sequences (at most length two) of $p_i$-gons between two consequtive $q$-gons at the $r $-th layer is of the same form as for $\mathfrak{k}_n$, hence it can be extended to a full neighborhood. Secondly, each partial neighborhood of a
\(q_i\)-gon occurring in this construction can first be extended
within \(\mathfrak{p}_n^{\circ}\) until its two ends have \(q_r\)
and \(q_s\) for some $r, s<m$. The segment $ [q_1,q_i,q_r], 
[q_1,q_i,q_s] $ in $E_n$ 
then connect both ends to \(q_1\), while repetitions of
$[q_1,q_i,q_1] $
fill the remaining positions. Thus, auxiliary tuple $E_n$ allows us to bypass reversal method of Lemma~\ref{even} used to extend  partial neighborhood around even-sided faces in $ \mathfrak{k}_n$ to a full neighborhood.

 Thus every
required \(q_i\)-neighborhood can be completed, and therefore
\[
\mathcal H(\mathfrak{p}_n)\geq n.
\]
Consequently,
\[
\mathcal H(\mathfrak{p}_n)=n.
\]

Let $Q_n$ be the canonical dual tile associated with
$p_{n+1}$. A corner of $Q_n$ associated with an $r$-gon in the original tiling has angle $2\pi/r$ and is referred to as a type-$r$ corner. A full configuration of copies of $Q_n$ meeting at a vertex of $Q_n$ is given by a finite choice of integers $k_1, \dots, k_m \in \mathfrak{p}_n$, possibly with repetitions, such that
		\[ \sum_{i=1}^{m} \frac{2\pi}{k_i} = 2\pi. \]
	
 Since all corner types \(k_i\) are primes, Lemma~\ref{reciprocal1} implies that they have the same type. In particular, a non-edge-to-edge $T$-junction cannot be formed by $Q_n$ to fully surround a vertex. Consequently, a tiling patch of $(r+1)$-layer by $Q_n$ must contain an $r$-layer patch as the dual of a $r$-layer homogeneous patch, and 
 \[ n \leq \mathcal H(Q_n) \leq n+1. \]

 The construction of $ \mathfrak{p}_n$ in fact gives an almost complete $(n+1)$-layer patch: the only obstruction is the terminal $p_{n+1}$-gon in $\mathcal C_n$, whose neighbourhood is complete except at two adjacent vertices in $\partial \mathcal C_n$. This is similar to the situation in Figure \ref{around7} around the $7$-gon for $n=2$. In the dual configuration this leaves, at the corresponding vertex $u$, a gap of angle $2\theta$, where $\theta=\frac{2\pi}{p_{n+1}}$. This gap can be filled by type-$p_{n+1}$ corners of two further copies of $Q_n$ together, as illustrated schematically in Figure~\ref{lasttwotiles} by the copies $R$ and $S$, each contributing a corner of angle $\theta$ to fill this gap. This is possible because $R$ and $S$ are not required to meet the remainder of the configuration edge-to-edge. Hence $Q_n$ admits an $n+1$-layer patch.

 Combining this construction with the preceding upper bound gives
\[ \mathcal H(Q_n)=n+1.\]
Relabelling 
$Q_{n-1}$ as $T_n$ proves the theorem for $n\geq2$; the case (n=1) is supplied by the base construction.

\begin{figure}[ht]
\centering
\begin{tikzpicture}[
	x=.06cm,
	y=-.06cm,
	line cap=round,
	line join=round,
	tile/.style={draw=originalblue,line width=.45pt},
	special/.style={draw=dualvermillion,line width=.65pt},
	level/.style={draw=originalblue,line width=.70pt},
	dots/.style={draw=originalblue,densely dotted,line width=.48pt},
	dualdots/.style={draw=dualvermillion,densely dotted,line width=.52pt},
	every node/.style={inner sep=.8pt}
	]
	% The three layer levels. Since the vertical TikZ scale is reversed,
	% smaller y-coordinates appear higher on the page.
	\def\ytop{114.5}
	\def\ymid{135.5}
	\def\ybot{165.7}
	
	% Horizontal layer endpoints.
	\coordinate (topL) at (49.27,\ytop);
	\coordinate (topR) at (199.68,\ytop);

	\coordinate (midL) at (57.24,\ymid);

	% Left dotted portion (chosen as mirror image of the right dotted part)
	\coordinate (midLD1) at (75.48,\ymid);
	\coordinate (midLD2) at (85.75,\ymid);

	% Right dotted portion
	\coordinate (midD1) at (156.33,\ymid);
	\coordinate (midD2) at (166.60,\ymid);

	\coordinate (midR) at (184.84,\ymid);
	\coordinate (botL) at (79.15,\ybot);
	\coordinate (botR) at (148.67,\ybot);
	
	% Principal vertices on the middle and bottom levels.
	\coordinate (mL)  at (86.611,\ymid);
	\coordinate (vi)  at (103.245,\ymid);
	\coordinate (vip) at (131.407,\ymid);
	\coordinate (mR1) at (148.882,\ymid);
	\coordinate (mR2) at (168.260,\ymid);
	\coordinate (bL1) at (93.248,\ybot);
	\coordinate (bR1) at (134.065,\ybot);
	
	% Central vertex and exact intersections of the trimmed lower rays.
	\coordinate (v)  at (116.386,150.879);
	\coordinate (jL) at (80.69,154.0);
	\coordinate (jR) at (160.8,156.0);
	
	% Vertices of the two copies R and S in the original schematic.
	\coordinate (rL) at (87.52,128.12);
	\coordinate (rT) at (109.11,118.73);
	\coordinate (rU) at (119.98,126.95);
	\coordinate (sU) at (147.58,126.710);
	\coordinate (sT) at (136.800,117.56);
	\coordinate (sL) at (120.85,120.85);

	% Extra vertices so that R and S become proper polygons
	\coordinate (rP1) at (93.0,120.3);
	\coordinate (rP2) at (101.2,116.8);
	\coordinate (sP1) at (148.6,120.3);
	\coordinate (sP2) at (144.3,117.0);
	
	% Other named vertices used by the fan around v.
	\coordinate (uL1) at (92.45,132.34);
	\coordinate (uL3) at (70.9,120);
	\coordinate (uR1) at (151.82,123.42);
	\coordinate (uR2) at (157.449,121.78);
	\coordinate (uR3) at (160.265,131.64);
	\coordinate (longR) at (182.84,133);
	
	% All three levels are exactly straight and horizontal.
	\draw[level] (topL)--(topR);

	% Middle level with symmetric dotted portions
	\draw[level] (midL)--(midLD1);
	\draw[dots]  (midLD1)--(midLD2);
	\draw[level] (midLD2)--(midD1);
	\draw[dots]  (midD1)--(midD2);
	\draw[level] (midD2)--(midR);

	\draw[level] (botL)--(botR);
	
	% Left part of the original schematic.
	\draw[tile] (68.6,\ymid)--(bL1);
	\draw[tile] (59.6,\ytop)--(mL);
	\draw[tile] (50.6,\ytop)--(mL);
	\draw[tile] (75.6,\ytop)--(mL);
	\draw[special] (uL3)--(uL1)--(v);
	
	% Lower rays terminate exactly on neighbouring polygon edges.
	\draw[special] (v)--(jL);
	\draw[tile] (midR)--(bR1);
	\draw[special] (v)--(jR);
	
	% Fans whose endpoints lie exactly on the top level.
	\draw[tile] (vi)--(93.072,\ytop);
	\draw[tile] (143.135,\ytop)--(vip);
	\draw[tile] (mR1)--(158.153,\ytop);
	\draw[tile] (mR1)--(164.020,\ytop);
	\draw[tile] (mR1)--(168.781,\ytop);
	
	% The two copies R and S as proper polygons; no edge is drawn twice.
	\draw[special] (uL1)--(rL);
	\draw[special] (rL)--(rP1)--(rP2)--(rT);
	\draw[special] (rT)--(rU);

	% The common R--S side is drawn once, from v through sL to sT.
	\draw[special] (v)--(sL)--(sT);
	\draw[special] (v)--(sU);
	\draw[special] (sU)--(sP1)--(sP2)--(sT);
	
	% Remaining rays of the original central fan.
	\draw[special] (sU)--(uR1)--(uR2);
	\draw[dualdots] (uR2)
	.. controls (160.669,121.798) and (162.846,126.593) .. (uR3);
	\draw[special] (v)--(uR3);
	\coordinate (rightFanMeet) at ($(mR2)!.242!(171.763,\ytop)$);
	\draw[dualdots] (uR3)--(rightFanMeet);
	\draw[special] (v)--(longR);
	
	% Dotted angular indications around v.
	\coordinate (a1) at ($(v)!.31!(rL)$);
	\coordinate (a2) at ($(v)!.266!(sL)$);
	\draw[dualdots] (a1)
	.. controls (110.5,141.7) and (114.2,141.6) .. (a2);
	
	\coordinate (a3) at ($(v)!.203!(sL)$);
	\coordinate (a4) at ($(v)!.222!(uR1)$);
	\draw[dualdots] (a3)
	.. controls (119.8,143.2) and (122.2,143.5) .. (a4);
	
	\coordinate (a5) at ($(v)!.335!(uR1)$);
	\coordinate (a6) at ($(v)!.332!(uR3)$);
	\draw[dualdots] (a5)
	.. controls (129.8,141.9) and (130.6,143.2) .. (a6);
	
	\coordinate (a7) at ($(v)!.22!(longR)$);
	\coordinate (a8) at ($(v)!.31!(jR)$);
	\draw[dualdots] (a7)
	.. controls (134.5,148.5) and (131.5,152.2) .. (a8);
	
	\coordinate (a9) at ($(v)!.36!(jL)$);
	\coordinate (a10) at ($(v)!.28!(rL)$);
	\draw[dualdots] (a9)
	.. controls (100.5,146.4) and (104.7,145.0) .. (a10);
	
	\coordinate (a11) at ($(v)!.40!(jL)$);
	\coordinate (a12) at ($(v)!.35!(jR)$);
	\draw[dualdots] (a11)
	.. controls (111.5,155.7) and (120.5,155.2) .. (a12);
	
	% Labels
	\node[anchor=east,font=\normalsize] at (47.8,\ytop)
	{$\partial X_{n+1}$};
	\node[anchor=east,font=\normalsize] at (55.8,\ymid)
	{$\partial X_n$};
	\node[anchor=east,font=\normalsize] at (77.8,\ybot)
	{$\partial X_{n-1}$};
	
	\node[font=\small] at (112.6,128.3) {$R$};
	\node[font=\small] at (126.0,127.4) {$S$};
	\node[font=\scriptsize] at (104.5,138.2) {$v_i$};
	\node[font=\scriptsize] at (129,137.6) {$v_{i\!{+}\!1}$};
	\node[font=\scriptsize] at (84.5,138.2) {$v_{i\!{-}\!1}$};
	\node[font=\scriptsize] at (149,137.6) {$v_{i\!{+}\!1}$};
	
	\node[font=\scriptsize] at (115.4,152.5) {$u$};
	
	\node[font=\small,text=dualvermillion] at (99.7,149.5) {$\theta$};
	\node[font=\small,text=dualvermillion] at (108.9,140.9) {$\theta$};
	\node[font=\small,text=dualvermillion] at (120.8,141.9) {$\theta$};
	\node[font=\small,text=dualvermillion] at (132.7,141.5) {$\theta$};
	\node[font=\small,text=dualvermillion] at (140,150.0) {$\theta=\frac{2\pi}{p_{n+1}}$};
	
	\node[font=\small] at (125.0,160.5) {{\large $p_{n+1}\text{-gon}$}};
	\node[font=\small,text=dualvermillion] at (115.2,156.5)
	{$\displaystyle \theta$};
\end{tikzpicture}
\caption{The local configuration around the terminal \(p_{n+1}\)-gon and its dual}
\label{lasttwotiles}
\end{figure}

\end{proof}

\begin{remark}[Euclidean all-but-one configurations] Motivated by the obstruction-propagation mechanism of Theorem~\ref{heeschn} of tuples with all-but-one extension property (\ref{allbutone}) at level $n$, we searched for such examples of Euclidean tiles with the same property. 
Among the polyominoes catalogued in Caplan's database~\cite{K21,K22}, we found a $9$-omino with $\mathcal H=1$ and a $16$-omino with $\mathcal H=2$ possessing this property at levels $1$ and $2$, respectively; see Figures~\ref{9-omino} and~\ref{fig:3369}. In particular, the $9$-omino realizes the Euclidean counterpart of the base configuration in Proposition~\ref{1butall}. These examples show that the required finite-stage configurations occur in the Euclidean plane, but they do not yet provide a procedure for passing from one level to the next. Such a procedure would have to complete the exceptional neighbourhood in one corona while forcing a new exceptional neighbourhood in the following corona, without interfering elsewhere.

	\definecolor{centralfill}{RGB}{62,78,92}
	\definecolor{firstfill}{RGB}{83,155,211}
	\definecolor{secondfill}{RGB}{248,176,79}
	
	\tikzset{
		cellgrid/.style={draw=white!72, line width=0.18pt},
		copyborder/.style={draw=black, line width=0.85pt,
			line cap=round, line join=round},
		copylabel/.style={font=\scriptsize\bfseries, text=black,
			fill=white, fill opacity=.72, text opacity=1,
			inner sep=1pt, rounded corners=.5pt}
	}
	
	% 9-OMINO DRAWING MACRO
	% #1 = TikZ scope options (e.g., cm or shift); #2 = fill colour; #3 = label.
	\newcommand{\DrawNine}[3][]{%
		\begin{scope}[#1]
			\foreach \x/\y in {0/2, 1/1, 1/2, 1/3, 2/0, 2/1, 2/3, 3/0, 3/3} {
				\path[fill=#2] (\x,\y) rectangle ++(1,1);
				\draw[cellgrid] (\x,\y) rectangle ++(1,1);
			}
			\draw[copyborder] (0,2)--(1,2)--(1,1)--(2,1)--(2,0)--(3,0)--(4,0)--
			(4,1)--(3,1)--(3,2)--(2,2)--(2,3)--(3,3)--(4,3)--
			(4,4)--(3,4)--(2,4)--(1,4)--(1,3)--(0,3)--cycle;
			\node[copylabel] at (2,2) {#3};
		\end{scope}%
	}
	
	% ----------------------------------------------------------------
	% FIGURE 1: 9-omino, first corona and all-but-one partial second corona.
	% ----------------------------------------------------------------
	\begin{figure}[H]
		\centering
		\begin{tikzpicture}[x=.35cm,y=.35cm]
			% Partial second corona: all copies have the same orange fill.
			\DrawNine[cm={ 1, 0, 0, 1,( 0,-5)}]{secondfill}{$S_1$}    % r0
			\DrawNine[cm={ 0,-1, 1, 0,( 4, 1)}]{secondfill}{$S_2$}    % r270
			\DrawNine[cm={ 0,-1, 1, 0,( 5, 7)}]{secondfill}{$S_4$}    % r270
			\DrawNine[cm={ 0, 1,-1, 0,(-4, 1)}]{secondfill}{$S_9$}    % r90
			\DrawNine[cm={ 0,-1,-1, 0,(-2, 3)}]{secondfill}{$S_{10}$}    % f90
			\DrawNine[cm={ 0, 1, 1, 0,( 3, 5)}]{secondfill}{$S_5$}    % f270
			\DrawNine[cm={-1, 0, 0,-1,(-1, 9)}]{secondfill}{$S_8$}    % r180
			\DrawNine[cm={ 1, 0, 0, 1,(-1, 5)}]{secondfill}{$S_6$}    % r0
			\DrawNine[cm={ 1, 0, 0, 1,( 6,-1)}]{secondfill}{$S_3$}    % r0
			\DrawNine[cm={ 0, 1, 1, 0,(-4, 7)}]{secondfill}{$S_{7}$}% f270
			
			% First corona: all six copies have the same blue fill.
			\DrawNine[cm={-1, 0, 0,-1,( 6, 3)}]{firstfill}{$T_1$}     % r180
			\DrawNine[cm={ 0, 1,-1, 0,( 8, 1)}]{firstfill}{$T_2$}     % r90
			\DrawNine[cm={ 0,-1,-1, 0,( 2, 3)}]{firstfill}{$T_6$}     % f90; exceptional
			\DrawNine[cm={ 0,-1, 1, 0,(-3, 7)}]{firstfill}{$T_4$}     % r270
			\DrawNine[cm={ 0, 1,-1, 0,( 0, 1)}]{firstfill}{$T_5$}     % r90
			\DrawNine[cm={-1, 0, 0,-1,( 5, 8)}]{firstfill}{$T_3$}     % r180
			
			% Central copy.
			\DrawNine[cm={ 1, 0, 0, 1,( 0, 0)}]{centralfill}{\textcolor{black}{$T_0$}}
		\end{tikzpicture}
		 \captionof{figure}{A $9$-omino with $\mathcal{H}=1$ and  "all but one property"}
\label{9-omino}
			\end{figure}

	\begin{figure}[H]
		\centering
		
		\centering
\includegraphics[scale=0.20]{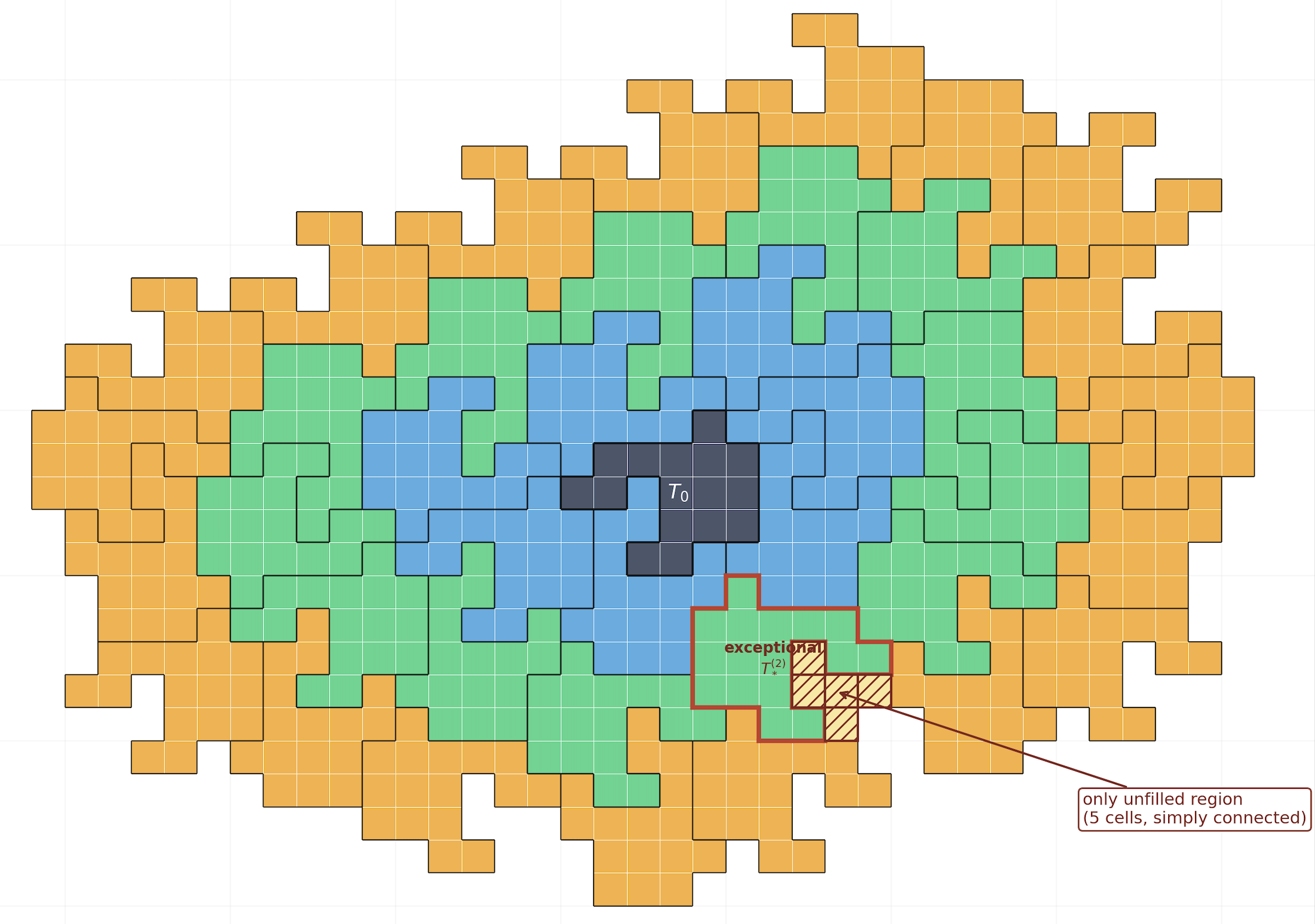}
\captionof{figure}{A $16$-omino with $\mathcal{H}=2$ and all-but-one extension property at level $2$}
		\label{fig:3369}

	\end{figure}

\end{remark}

	\section{Weakly Aperiodic Vertex Types} \label{aperiodic}
	In \cite{AM23}, the present author constructed a set of four regular polygons of sizes $\{3,5,k_3,k_4\}$, with $k_3 \neq k_4$ and $k_3,k_4 \geq 12$, that admit no periodic tiling. The proof relied on a specific double-counting argument to establish the aperiodicity of this tile set. To adapt this method to the homogeneous tiling setting and obtain an aperiodic cyclic tuple, we consider
	
	\[ \mathfrak{k}_a= [3, 5, k, 5, l, 5, m, 5, l, 5, k, 5, l, 5]\] 
	for distinct integers $k, l, m \geq 6$. 
	\\
	A tiling of type $\mathfrak{k}_a$ can be constructed by the inductive corona formation method. To extend the layer $X_k$ to $X_{k+1}$, one can first complete the neighborhoods around the triangles, next the pentagons and finally the $k$-, $l$,- and $m$-gons. Since the triangles and $k$-, $l$- and $m$-gons share edges only with pentagons in $\mathfrak{k}_a$, no constraint arises in forming neighborhoods around them at any stage of the construction. After the construction of neighborhoods around the triangles, the only possible partial neighborhood around a pentagon is $[3, k, l]$. Such a partial neighborhood can be extended to full neighborhood of type $[3, k, l, k, l]$ or $[3, k, l, m, l]$. Hence, there exists a tiling of type $\mathfrak{k}_a$.
	\par
	The following observations hold for any homogeneous tiling of this type:
	\par
	i) Each triangle shares common edges with three pentagons, and meets fifteen other pentagons only at vertices---five pentagons with each vertex of the triangle.
	\par
	ii) Each pentagon shares at least one common edge (at most three vertex-only incidences) with a triangle; otherwise one cannot form a valid neighborhood of the pentagon.
	\par
	In a strongly periodic tiling, Observation~(i) implies that the ratio of the total number of common edges to the total number of common single vertices between the triangles and the pentagons is $1:5$. On the other hand, Observation~(ii) forces this ratio to be $r:3$ for some $r \geq 1$. This inconsistency yields a contradiction, and hence these types are weakly aperiodic.

	\subsection{Weakly Aperiodic Convex Monotiles} Assume now that $3,5,k,l$ and $m$ are distinct primes, and let $T_a$ be the canonical dual of a fan of type $\mathfrak{k}_a$. Since a homogeneous tiling of type $\mathfrak{k}_a$ exists, its dual gives a tiling of $\mathbb H^2$ by congruent copies of the convex tile $T_a$.	
By a similar argument to that used in Theorem \ref{convexthm}, any full tiling by $T_a$ must be the dual of a homogeneous tiling of type $\mathfrak{k}_a$ even without any edge-to-edge restriction. Indeed, Lemma~\ref{reciprocal1} implies that every completed vertex is formed by corners of a single type. Since the dual construction is equivariant under isometries, $T_{a}$ is weakly aperiodic. Varying the distinct primes $k,l$ and $m$ produces infinitely many such tiles. Their interior angles, and consequently their areas, are rational multiples of $\pi$.

	\section*{Acknowledgement} The author would like to thank Subhojoy Gupta and Basudeb Datta for their valuable suggestions. He especially thanks Subhojoy Gupta for identifying a gap in the proof of the main result in the first draft of the manuscript. The research was supported by the Seed Grant DoRDC/730 of TIET for the year 2024-25. 
	\bibliographystyle{amsplain}
	\bibliography{bibtile}
	
\end{document}